\documentclass[pdflatex,sn-mathphys-num]{sn-jnl}

\usepackage{float}
\usepackage{tikz,pgf}
\usepackage{comment}
\usepackage{epstopdf}
\usepackage{enumerate}
\usepackage{mathrsfs}
\usepackage{anyfontsize}
\usepackage{amsthm}

\usepackage{array}
\usepackage{enumitem}
\usepackage[normalem]{ulem}
\usepackage{bookmark}
\usepackage{lscape}
\usepackage{pdflscape}

\usepackage{todonotes}
\newcommand{\nret}[1]{\sharp({#1})}

\usepackage{graphicx}%
\usepackage{multirow}%
\usepackage{amsmath,amssymb,amsfonts}%
\usepackage{amsthm}
\usepackage[title]{appendix}%
\usepackage{xcolor}%
\usepackage{textcomp}%
\usepackage{manyfoot}%
\usepackage{booktabs}%
\usepackage{algorithm}%
\usepackage{algorithmicx}%
\usepackage{algpseudocode}%
\usepackage{listings}%

\theoremstyle{thmstyleone}%
\newtheorem{theorem}{Theorem}

\theoremstyle{thmstyletwo}%
\newtheorem{remark}{Remark}%
\newtheorem{lemma}{Lemma}
\newtheorem{corollary}{Corollary}
\theoremstyle{thmstylethree}%

\begin{document}

\title{Characterisations of Planar Galled Networks}

\author*[1]{\fnm{Hexuan} \sur{Liu}}\email{lhx@ruri.waseda.jp}

\author*[2]{\fnm{Taoyang} \sur{Wu}}\email{taoyang.wu@uea.ac.uk}

\author[3]{\fnm{Guan-Ru} \sur{Yu}}\email{gryu@math.nsysu.edu.tw}

\affil*[1]{\orgdiv{Department of Pure and Applied Mathematics}, \orgname{Waseda University}, \orgaddress{\city{Tokyo}, \postcode{169-8555}, \country{Japan}}}

\affil[2]{\orgdiv{School of Computing Sciences}, \orgname{University of East Anglia}, \orgaddress{ \city{Norwich}, \postcode{NR4 7TJ}, \country{UK}}}

\affil[3]{\orgdiv{Department of Applied Mathematics}, \orgname{ National Sun Yat-sen University}, \orgaddress{\city{Kaohsiung}, \postcode{804},  \country{Taiwan}}}

\abstract{Rooted phylogenetic networks are widely used to represent the evolution of species that have undergone reticulate processes. However, these networks can be highly non-planar, making them more difficult to visualise and interpret than evolutionary trees. In this paper, we investigate planarity properties of galled networks, an important subclass of phylogenetic networks. We show that all planar galled networks are necessarily upward planar. Furthermore, by leveraging recent results on planar phylogenetic networks, we provide three characterisations for each of the outerplanar and terminal planar galled network classes in terms of forbidden vertex configurations, forbidden directed subgraphs, and forbidden structures in their associated underlying undirected graphs. These results contribute to a deeper understanding of the structural properties of galled networks and may inform future methods for their construction and visualisation.}


\maketitle

\begin{center} 
\emph{Dedicated to Andreas Dress, with gratitude for his mentorship and contributions to phylogenetics community.
}
\end{center}

\section{Introduction} \label{sec:intro}

Phylogenetic networks are directed acyclic graphs used to represent reticulate evolutionary histories involving processes such as hybridization, recombination, and lateral gene transfer \cite{Dress2010,HusonBook,Steel,KongEtAl,LinzWicke}. Planarity plays a central role in the visualisation of such networks, as planar drawings allow for clearer biological 
interpretation \cite{MoultonWu2022}. Indeed, several algorithms for constructing and drawing phylogenetic networks are specifically designed to produce planar outputs \cite{DressHuson2004, GambetteHuson2008, SpillnerNguyenMoulton2011}. In particular, galled networks \cite{HusonKlopper2007} form one of the most natural and well-studied subclasses of phylogenetic networks, in which each reticulation event, such as a hybridization or recombination, can be attributed to a single local cycle-like structure in the network.

Planarity is one of the most classical structural properties in graph theory, with deep connections to graph drawing, network visualisation, and combinatorial optimization. For undirected graphs, planarity is well-understood: efficient algorithms exist for planarity testing \cite{HoTa}, and Kuratowski's theorem \cite{Kuratowski} provides a complete structural characterisation via forbidden subdivisions. For directed graphs, however, the 
picture is far more complex. Edge orientations give rise to multiple distinct planarity notions that do not coincide in general, and forbidden-structure characterisations analogous to Kuratowski's theorem remain elusive for several important classes of digraphs, including upward planar digraphs.

For undirected graphs, classical forbidden-subgraph results characterise the relationship between planarity and certain forbidden structures \cite{Sys, Kuratowski}. For directed graphs, several distinct planarity notions have been studied. For example, a digraph is \emph{upward planar} if it admits a planar drawing in which every edge is directed monotonically upward \cite{HuLu, Thomassen1989, GaTa}. For phylogenetic networks,  Moulton and Wu \cite{MoultonWu2022} 
introduced the notion of \emph{terminal planarity} for networks that admit a planar drawing in which all terminals (the root and leaves) lie on the outer face. They established an equivalence between terminal planarity and the upward planarity of an associated supergraph (the so-called completion), showed that planar, upward planar, terminal planar, and outerplanar networks form a strict hierarchy, and gave polynomial-time algorithms for deciding membership in each planarity class. 
More recently, Miyaji et al.~\cite{Miyaji2026} showed that general terminal phylogenetic networks can be characterised by a set of forbidden structures consisting of six families of 0/1 labeled graphs.

\begin{figure}[h]
    \centering
    \includegraphics[scale=0.55]{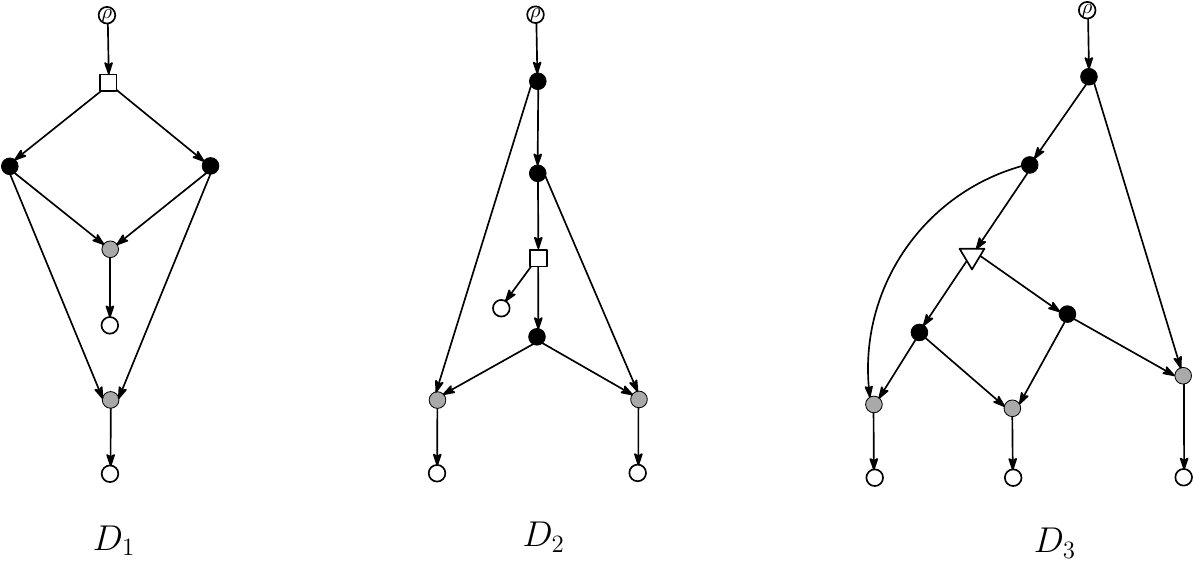}
    \caption{The three forbidden galled networks for outerplanar galled networks. Note that  $D_1$ and $D_2$ are the two forbidden structures for terminal galled networks. The square tree vertices in $D_1$ and $D_2$ are primitive, and the inverted triangular tree vertex in $D_3$ is separable.}
    \label{fig:mainforbidden}
\end{figure}

In this paper, we study rooted binary galled networks. We begin by providing a complete proof for the claim that planarity coincides with upward planarity for galled networks (Theorem~\ref{upward:thm}), which is stated without proof by Tollis and Kakoulis~\cite{TollisKakoulis2020}.
Next, we provide three equivalent characterisations of outerplanarity and terminal planarity in the context of galled networks. First, by exploiting a key structural property of galled networks, we reduce outerplanarity to the absence of a particular type of subdivision of $K_{2,3}$, which we call a normal subdivision (Theorem~\ref{thm:normalk23}). Similarly, terminal planarity is characterised by the absence of a more restricted subclass of such normal subdivisions (Theorem~\ref{thm:normal-Kk23}).
Second, in Theorem~\ref{noter2gall}, we show that a galled network is terminal planar if and only if it contains no primitive tree vertex (an internal tree vertex lying on the intersection path of two galled cycles). Moreover, it is outerplanar if and only if it contains neither a primitive tree vertex nor a separable tree vertex (that is, a tree vertex $v$ contained in exactly three galled cycles, each involving a distinct pair of neighbours of $v$); see Theorem~\ref{thm:outerplanar-forbidden} and Fig.~\ref{fig:mainforbidden}.
Finally, Theorem~\ref{outer forbidden} establishes that a galled network is outerplanar if and only if it contains no directed subdivision of any of the three minimal forbidden galled networks $D_1$, $D_2$, and $D_3$ shown in Fig.~\ref{fig:mainforbidden} as a subgraph. In contrast, terminal planarity is characterised in Theorem~\ref{terminal forbidden} by the exclusion of only $D_1$ and $D_2$, which leverages a recent characterisation of general terminal phylogenetic networks by Miyaji et al.~\cite{Miyaji2026}.

The remainder of the paper is organised as follows. Section~\ref{sec:background} introduces the basic definitions and notation of graph theory and phylogenetic networks. Section~\ref{sec:galled} reviews the definition of galled networks and provides some basic properties for such networks. Section~\ref{sec:upward} establishes the relationship between planarity and upward planarity for galled networks. Sections~\ref{sec:outer} and~\ref{sec:terminal} characterise outerplanar and terminal planar galled networks, respectively. Section~\ref{sec:conclusion} concludes with a brief discussion. For completeness, an appendix is included on the binary case of the recent characterisation of general terminal phylogenetic networks by Miyaji et al.~\cite{Miyaji2026}.

\section{Preliminaries}\label{sec:background}

\subsection{Graph}
We assume that all graphs in this paper are finite and simple, containing neither self-loops nor parallel edges. Let $G$ be an undirected graph with vertex set $V(G)$ and edge set $E(G)$, abbreviated as $V$ and $E$ when $G$ is clear from the context. For a vertex $v \in V(G)$, we denote by $\deg_G(v)$ the number of edges of $G$ incident to $v$. We write $\{u,v\}$ for an undirected edge, that is, an unordered pair of vertices, and $(u,v)$ for a directed edge from $u$ to $v$, that is, an ordered pair. Depending on the context, $\{u,v\}$ may denote either this edge or the two-element set of vertices $u$ and $v$. For a directed graph $G$, its underlying graph, denoted by $\overline{G}$, is the undirected graph obtained from $G$ by replacing each directed edge $(u,v)$ with the undirected edge $\{u,v\}$. A graph $H$ is a \emph{subgraph} of a graph $G$ if $V(H)\subseteq V(G)$ and $E(H)\subseteq E(G)$, with the orientations of edges inherited from $G$ when $G$ is directed. If $V(H)=V(G)$ and $E(H)=E(G)$ hold, we write $H=G$.

A \emph{walk} $W = (v_0, v_1, \dots, v_k)$ in $G$ is a finite sequence of vertices such that $\{v_{i-1}, v_i\} \in E(G)$ for each $1 \le i \le k$. Since all graphs considered here are simple, a vertex sequence uniquely determines the traversed edges. A walk is \emph{closed} if $v_0 = v_k$, and \emph{simple} if no vertex appears more than once, except possibly $v_0 = v_k$. In general, internal vertices $v_1, \dots, v_{k-1}$ of a walk may repeat. 
A \emph{subwalk} of $W$ between $v_i$ and $v_j$, where $0 \leq i<j \leq k$, is the contiguous subsequence $(v_i,v_{i+1},\dots,v_j)$.

A \emph{path} $P=(v_0, v_1, \dots, v_k)$ is a simple walk, and the \emph{subpath} of $P$ between $v_i$ and $v_j$ is denoted as $P[v_i,v_j] = (v_i, \dots, v_j)$. Two paths are \emph{internally disjoint} if they share no internal vertices. Two paths are \emph{vertex-disjoint} if they share no  vertices. A \emph{cycle} $C$ is a closed path, that is, a simple closed walk containing at least three vertices. If $G$ is a directed graph, the terms \emph{walk}, \emph{path}, \emph{cycle}, \emph{subpath}, and \emph{degree} refer to the corresponding objects in the underlying  graph $\overline{G}$; in particular, $\deg_G(v):=\deg_{\overline{G}}(v)$. A \emph{directed walk} (resp.\ \emph{directed path}, \emph{directed cycle}) is a walk (resp.\ path, cycle) $(v_0, v_1, \dots, v_k)$ in $\overline{G}$ such that $(v_{i-1}, v_i) \in E(G)$ for every $1 \le i \le k$.

Furthermore, a \emph{subdivision} of an undirected graph $G$ is a graph $G'$ obtained from $G$ by replacing the edges of $G$ with pairwise internally disjoint paths of non-zero length between the corresponding endpoints. The vertices of $G'$ corresponding to vertices of $G$ are referred to as the \emph{branch vertices} of $G'$.
For each edge $\{u, v\}$ of $G$, the path in $G'$ that replaces $\{u, v\}$ is called the \emph{canonical path} between $u$ and $v$ in $G'$. A \emph{directed subdivision} of a directed graph $G$ is defined analogously, replacing each directed edge $(u,v) \in E(G)$ with a directed path from $u$ to $v$ of non-zero length.

A graph $G$ is \emph{connected} if there is a path between any two vertices of $G$. A vertex $v$ of a graph $G$ is a \emph{cut vertex} if there exist two vertices $x,y \in V(G)\setminus\{v\}$ such that $G$ contains a path between $x$ and $y$, and every path between $x$ and $y$ in $G$ contains $v$. An edge $e$ of $G$ with endpoints $x$ and $y$ is a \emph{cut edge} if every path between $x$ and $y$ in $G$ uses the edge $e$. A graph $G$ is \emph{$k$-connected} if it has more than $k$ vertices and remains connected after the removal of any set of fewer than $k$ vertices. A connected graph is \emph{biconnected} if it has no cut vertex. A \emph{block} of $G$ is a maximal connected subgraph without a cut vertex; thus, every block is a maximal biconnected subgraph, that is, either a maximal $2$-connected subgraph, a cut edge with its two endpoints, or an isolated vertex.

Finally, we introduce the concept of a pivotal basis for a subgraph, which is motivated by Miyaji et al.~\cite{Miyaji2026} and will be used primarily in Section~\ref{sec:terminal}. 
Consider a subgraph $H$ of a graph $G$. A \emph{pivotal ray} of $H$ in $G$ is a path $P$ between a vertex $v\in V(H)$ and a cut vertex $u$ of $G$ such that $V(P) \cap V(H) = \{v\}$ and $u$ is the only cut vertex of $G$ on $P$. 
If such a ray exists, we call $v$ \emph{cut-visible} from $H$ in $G$.  Note that the ray $P$ is allowed to be degenerate, that is, $v$ is itself a cut vertex and $P$ consists of the vertex $v$ only.
A \emph{pivotal basis} of $H$ in $G$ is a set of vertices
$B = \{v_1, \dots, v_k\} \subseteq V(H)$ such that there exist pairwise vertex-disjoint pivotal rays $P_1, \dots, P_k$ in $G$ with
$V(P_i) \cap V(H) = \{v_i\}$ for each $1 \le i \le k$. Note that this implies that no two distinct rays $P_i$ and $P_j$ share a common vertex. See Figure~\ref{fig:cut_visible_illustration} for examples and a counterexample. 
Note that if $H=G$, then a pivotal basis of $H$ consists of cut vertices in $G$.

\begin{figure}[htbp]
    \centering
    \includegraphics[width=0.8\textwidth]{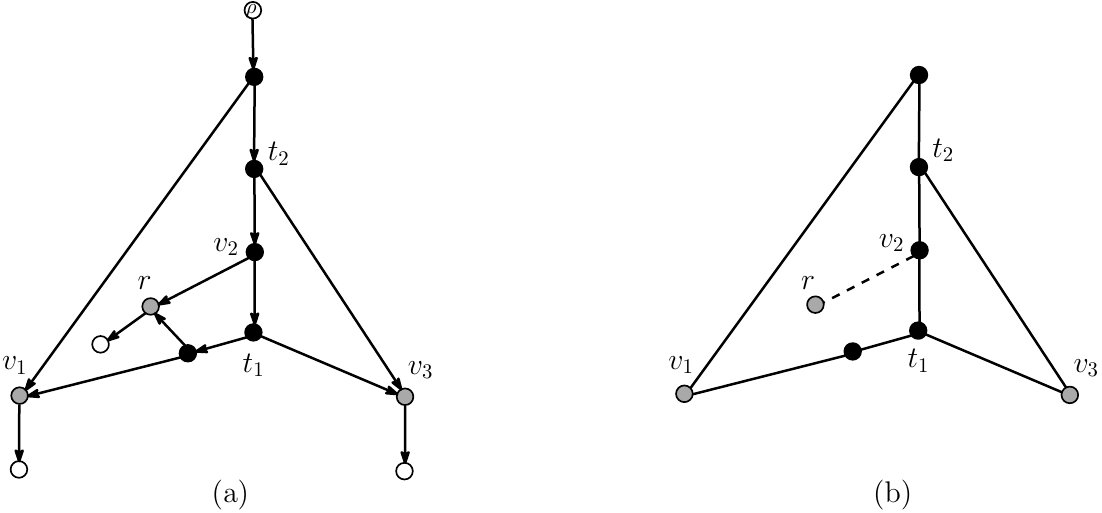}
    \caption{An illustration of pivotal bases. (a) A phylogenetic network $N$. (b) A subdivision $H$ of $K_{2,3}$ in the underlying graph $\overline{N}$ of $N$, shown by solid lines. Note that pivotal bases need not be unique. For example, both $\{v_1,v_2,v_3\}$ and $\{v_1,v_2\}$ are pivotal bases of $H$ in $\overline{N}$. However, $\{t_1,t_2\}$ is not a pivotal basis, since neither $t_1$ nor $t_2$ is cut-visible from $H$ in $\overline{N}$.}

    \label{fig:cut_visible_illustration}
\end{figure}

\subsection{Phylogenetic network}

Given a finite set $X$ of labels, a rooted and binary phylogenetic network (which from now on we will just call a \emph{phylogenetic network}) on $X$ is an acyclic directed graph $N=(V,E)$ with $X \subseteq V$, a unique root $\rho$ with in-degree $0$ and out-degree $1$, and all other non-root vertices belonging to one of the following three types: vertices in $X$, with in-degree $1$ and out-degree $0$, called \emph{leaves}; 
vertices with in-degree $1$ and out-degree $2$, called \emph{tree vertices}; and vertices with in-degree $2$ and out-degree $1$, called \emph{reticulation vertices}.
For a directed subgraph $H$ of $N$, we denote the underlying  graph of $H$ by $\overline{H}$. 
We use $\nret{H}$ to denote the number of reticulation vertices in  $H$, and put  $\nret{\overline{H}}=\nret{H}$.

Let $N$ be a phylogenetic network. If there exists a directed path $P$ in $N$ from $u$ to $v$, then $u$ is an \emph{ancestor} of $v$ in $N$ and $v$ is a \emph{descendant} of $u$ in $N$, denoted as $v<u$, or $v\le u$ if $u=v$ is allowed.
If $P$ is a directed edge, we also say that $u$ is a \emph{parent} of $v$ and $v$ is a \emph{child} of $u$.  If neither vertex is an ancestor of the other, they are said to be \emph{incomparable} $N$, denoted as $u \parallel v$. The qualifier ``in $N$'' is omitted when the network is clear from the context.

A path in $N$ is called a \emph{tree path} if each vertex on this path is a tree vertex of $N$; a \emph{directed tree path} is defined analogously. An edge incident with two tree vertices in $N$ is referred to as a \emph{tree edge}. The following lemma includes some elementary facts that are useful in later sections, its proof is straightforward and hence omitted.

\begin{lemma}\label{lmm:cycle}
Let $N$ be a phylogenetic network. Then every cycle of $N$ contains at least one tree vertex and at least one reticulation vertex. Moreover, for any two tree vertices $u$ and $v$ of $N$, there is at most one path between $u$ and $v$ whose vertices are all tree vertices.
\end{lemma}

The next lemma shows an elementary structure of a tree path.

\begin{lemma}\label{lem:treepath-top} 
Let $P$ be a tree path in a phylogenetic network $N$. Then $P$ has a unique vertex $s$ with no incoming edge in $P$.
Moreover, $s$ is an ancestor of every other vertex of $P$, and $P$ is either a directed path in $N$ starting at $s$ or the union of two such directed paths.
\end{lemma}

\begin{proof} Suppose that $P$ has $k$ edges, then it has $k+1$ vertices. Each edge of $P$ is directed into exactly one of its two endpoints.
Since all vertices on $P$ are tree vertices, each has at most one incoming edge in $P$.
Thus the $k$ edges of $P$ are incoming edges of $k$ distinct vertices on $P$, so there is a unique vertex $s$ with no incoming edge in $P$.

If $s$ is an endpoint of $P$, then every edge of $P$ is directed away from $s$, so $P$ is a directed path in $N$ starting at $s$. If $s$ is not an endpoint of $P$, the same argument applies to the two subpaths of $P$ between $s$ and the two endpoints of $P$, and hence $P$ is the union of two directed paths in $N$ starting at $s$. In both cases, $s$ is an ancestor of every other vertex of $P$. \end{proof}

We shall refer to the unique vertex $s$ in Lemma~\ref{lem:treepath-top} as the \emph{highest vertex} of $P$.

\subsection{Planar Graph}
Given a graph $G=(V,E)$, a drawing $\Gamma(G)$ of $G$ maps each vertex $v\in V$ to a distinct point $\Gamma(v)\in\mathbb{R}^2$, called a \emph{node}; we write $\Gamma$ for $\Gamma(G)$ when $G$ is clear from the context. Each edge $e\in E$ with endpoints $u$ and $v$ is drawn as a Jordan curve between $\Gamma(u)$ and $\Gamma(v)$, meaning a non-self-intersecting curve with these two nodes as its endpoints, called the \emph{arc} $\Gamma(e)$, oriented from $\Gamma(u)$ to $\Gamma(v)$ when $G$ is directed and $e=(u,v)$. An arc passes through no node other than its two endpoints. More generally, for a path $P$ we write $\Gamma(P)$ for its \emph{arc}, the concatenation of the arcs of its edges. The positions of the nodes and the shapes of the arcs may be chosen freely, provided that the resulting picture unambiguously represents the graph.
Moreover, if $H$ is a subgraph of $G$, the \emph{restriction} of a drawing $\Gamma$ of $G$ to $H$ is the drawing of $H$ obtained by keeping the nodes $\Gamma(v)$ with $v\in V(H)$ and the arcs $\Gamma(e)$ with $e\in E(H)$.

A drawing $\Gamma$ is \emph{planar} if no two distinct arcs intersect except possibly at common nodes, and a graph is \emph{planar} if it admits a planar drawing. Note that, in a planar drawing, if $P$ is a path with $m$ vertices, then $\Gamma(P)$ is a Jordan curve containing $m$ nodes.

Given a subset $A\subseteq \mathbb{R}^2$, a \emph{region} of $A$ is a maximal connected open subset of $A$.  In particular, for a planar drawing $\Gamma(G)$ of $G$, the regions of $\mathbb{R}^2 \setminus \Gamma(G)$ are referred to as the \emph{faces} of $\Gamma$, denoted by $F(\Gamma(G))$. Note that there is exactly one unbounded face in $F(\Gamma(G))$, called the \emph{outer face}; all other bounded faces are \emph{internal faces}. The boundary of a face $f$ is denoted $\partial f$. For every cycle $C$ in $G$, the image $\Gamma(C)$ is a closed Jordan curve. By the Jordan curve theorem (cf.~\cite[Theorem 4.1.1]{Diestel}), 
the set $\mathbb{R}^2 \setminus \Gamma(C)$ has exactly two regions: one is  bounded and denoted by $\mathrm{int}(C)$; the other is unbounded and denoted by $\mathrm{ext}(C)$. Note that each of them has $\Gamma(C)$ as its boundary. For later use, we denote their closures as $\mathrm{Int}(C) = \mathrm{int}(C) \cup \Gamma(C)$ 
and $\mathrm{Ext}(C) = \mathrm{ext}(C) \cup \Gamma(C)$, respectively.
Furthermore, we recall the following corollary of the Jordan curve theorem:

\begin{lemma}[{\cite[Lemma~4.1.2(i)]{Diestel}}]\label{lmm:diestel-412}
Let $\gamma_1, \gamma_2, \gamma_3$ be three Jordan curves between the same two endpoints but otherwise disjoint. Then $\mathbb{R}^2\setminus (\gamma_1\cup \gamma_2 \cup \gamma_3)$ has exactly three regions, with boundaries $\gamma_1\cup \gamma_2$, $\gamma_2\cup \gamma_3$, and $\gamma_1\cup \gamma_3$.
\end{lemma}

The \emph{boundary walk} of $f \in F(\Gamma)$ is a closed walk $W$ in $G$ satisfying $\Gamma(W) = \partial f$, up to cyclic rotation and reversal. Since a cut edge of $N$ whose arc lies on $\partial f$ is traversed twice in $W$, $W$ may contain repeated vertices or edges. A face $f$ is a \emph{simple face} if $\partial f$ is a closed Jordan curve; equivalently, $f$ is a simple face if and only if its boundary walk $W$ is a cycle in $G$.

Two planar drawings $\Gamma_1$ and $\Gamma_2$ of a graph $N=(V,E)$ are \emph{strongly equivalent} if there exists a bijection $\sigma:F(\Gamma_1)\to F(\Gamma_2)$ such that $\sigma$ maps the outer face
of $\Gamma_1$ to the outer face of $\Gamma_2$ and, for every $x\in V\cup E$ and every $f\in F(\Gamma_1)$, the image $\Gamma_1(x)$ lies on $\partial f$ if and only if $\Gamma_2(x)$ lies on $\partial(\sigma(f))$.

A planar graph is \emph{outerplanar} if it admits a planar drawing in which every node lies on the boundary of the outer face. A phylogenetic network $N$ on a leaf set $X$ with root $\rho$ is \emph{terminal planar}~\cite{MoultonWu2022} if it admits a planar drawing in which the nodes corresponding to the set $\{\rho\} \cup X$ all lie on the boundary of the outer face. A planar drawing of a directed graph is an upward planar drawing if each arc is monotonically non-decreasing in the $y$-coordinate, and a directed graph is \emph{upward planar}~\cite{Thomassen1989} if it admits such a drawing. A phylogenetic network is outerplanar (resp. terminal planar, planar) if and only if its underlying graph possesses this property.
For phylogenetic networks, outerplanarity is strictly stronger than terminal planarity, which is in turn strictly stronger than upward planarity, and upward planarity is strictly stronger than planarity~\cite{MoultonWu2022}. Furthermore, the property of being planar, outerplanar, or upward planar is closed under taking subgraphs.

\section{Galled Networks}  \label{sec:galled}

A \emph{galled cycle} in a phylogenetic network $N$ is a subgraph $G$ consisting of two internally disjoint directed paths from a common tree vertex, called the \emph{top vertex} of $G$, to a common reticulation vertex $r$, such that every vertex on these two paths other than $r$ is a tree vertex.
Two distinct galled cycles $G_1$ and $G_2$ are said to \emph{overlap} if $V(G_1) \cap V(G_2) \neq \emptyset$. 

We use the same symbol for a galled cycle and its underlying graph when no confusion arises. A \emph{galled network} $N$ is a phylogenetic network in which every reticulation vertex is contained in a unique galled cycle (see Figure~\ref{fig:mainforbidden} for examples, all three phylogenetic networks are galled networks).

Lemma~\ref{lmm:retcutvertex} is an immediate consequence of a result of Gunawan, Rathin, and Zhang~\cite[Theorem~2.1(2)]{GuRaZh}.

\begin{lemma}\label{lmm:retcutvertex}
Let $N$ be a galled network, and let $r$ be a reticulation vertex with child $v$. Then $r$ is a cut vertex of $\overline{N}$, and the edge $\{r,v\}$ is a cut edge of $\overline{N}$.
\end{lemma}

\begin{lemma}\label{lmm:deg3-tree}
Let $N$ be a galled network and let $B$ be a 2-connected subgraph of 
$\overline{N}$. Then every vertex 
$v \in V(B)$ with $\deg_B(v) = 3$ is a tree vertex of $N$.
\end{lemma}

\begin{proof}
Suppose for a contradiction that $v$ is a reticulation vertex $r$. Then all three edges incident to $r$ in $\overline{N}$ lie in $E(B)$, including the outgoing edge $e$, which is a cut edge of $\overline{N}$ by Lemma~\ref{lmm:retcutvertex}; hence $e$ lies on no cycle, contradicting the assumption that $B$ is 2-connected. Since the root and leaves have degree $1$, and $v$ is not a reticulation vertex, $v$ must be a tree vertex.
\end{proof}

We now turn to the main structural property of cycles in galled networks.

\begin{theorem}\label{cycleandgalledcycle}
Let $C$ be a cycle in a galled network $N$ and suppose that $v$ is a vertex of $C$. Then there exists a reticulation vertex $r$ on $C$ such that $v$ is contained in the galled cycle $G_r$ associated with $r$. 
In particular, if $C$ contains exactly one reticulation vertex $r$, then 
$C=G_r$.
\end{theorem}

\begin{proof}
If $v$ is a reticulation vertex on $C$, then $v$ is contained in its own galled cycle. Hence we may further assume that $v$ is a tree vertex. We argue by induction on $k=\nret{C}$, the number of reticulation vertices on the cycle $C$. 

For the case $k=1$,  every vertex of $C$ other than its reticulation vertex $r$ is a tree vertex. 
Removing $r$ from $C$ yields a tree path $P$ between the two parents of $r$. By Lemma~\ref{lem:treepath-top}, 
there exists a highest vertex $s$ in $P$ which is an ancestor of every other vertex in $P$. This implies that 
$C$ is the union of the two internally disjoint directed paths in $N$ from $s$ to $r$. Thus $C=G_r$.

Now consider the case $k\ge 2$, with the induction assumption that the statement holds for cycles containing fewer than $k$ reticulation vertices.
Denote the parent of $v$ by $v_0$ and its two children by $v_1$ and $v_2$. As $C$ contains two vertices from $\{v_0,v_1,v_2\}$, swapping $v_1$ and $v_2$ if needed, we may assume that $v_1$ is contained in $C$. Denote the transversal of $C$ starting from the directed edge $(v,v_1)$ by $[u_0:=v,u_1:=v_1,u_2,\dots,u_m,u_{m+1}]$ for some $m\ge 2$, using the convention that $u_{m+1}=u_0$. 
For $0\leq p\le q \leq m+1$, let $U[p,q]$ be the subpath $[u_p,u_{p+1},\dots,u_q]$ of $C$ between $u_p$ and $u_q$, with $U[p,p]=u_p$. In particular, $U[0,m+1]=C$.

As $N$ is acyclic, $C$ is not a directed cycle, hence there exists a (necessarily unique) index $j\in\{1,\dots,m\}$ such that $(u_{i-1},u_i)\in E(N)$ for every $1\leq i\leq j$, while $(u_{j+1},u_j)\in E(N)$.
Consequently, the vertex $r := u_j$ has two parents  $u_{j-1}$ and $u_{j+1}$, and hence is a reticulation vertex in $N$. Furthermore, $U[0,j-1]$ is a tree path. Otherwise, if some vertex of $U[0,j-1]$ other than $u_0$ were a reticulation vertex, then its outgoing edge would lie on $C$, contradicting Lemma~\ref{lmm:retcutvertex}.
Finally, note that $u_0$ is the unique highest vertex in  $U[0,j-1]$.  

Consider the galled cycle $G_r$ associated with $r$, whose top vertex is denoted by $s$. Let $P_r$ and $Q_r$ be the two directed paths from $s$ to $r$ in $G_r$ so that  $u_{j-1}$, one parent of $r$, is contained in $P_r$ while the other parent $u_{j+1}$ is contained in $Q_r$.
We distinguish two cases according to the position of $s$.

\smallskip
\noindent
\textit{Case 1. $s\notin \{u_1,\dots,u_{j-1}\}$.}

Since $U[0,j]$ is a directed path ending at the reticulation vertex $r$ and whose internal vertices are tree vertices, its last edge $(u_{j-1},r)$ lies on $P_r$. As $s\notin \{u_1,\dots,u_{j-1}\}$, whole $U[0,j]$ is contained in $P_r$. Hence $v=u_0\in V(G_r)$, and the conclusion holds.

\smallskip
\noindent
\textit{Case 2. $s\in \{u_1,\dots,u_{j-1}\}$.}

Let $s = u_a$ for some $1 \leq a \leq j-1$. Then $U[a,j]$ is coincident with $P_r$ as they have the same set of vertices. Furthermore, as $u_0$ is the highest vertex of $U[0,j-1]$, it is an ancestor of $s=u_a$. Traversing $C$ from $u_0$ along $[u_0,u_m,u_{m-1},\dots,u_{j+1},u_j]$, let $u_b$ be the first vertex in this traversal 
that belongs to $Q_r$. Then
we have $j+1\le b\le m$ because $u_{j+1}\in Q_r$ while $u_0\not \in Q_r$ (as it is an ancestor of $s$).

Let $Q_r[u_a,u_b]$ denote the subpath of $Q_r$ between $u_a$ and $u_b$ and let the path $U[0,a]$ be a subpath of the tree path $U[0,j-1]$. Both $U[0,a]$ and $Q_r[u_a,u_b]$ are tree paths.  Set $C'=U[b,m+1]\cup U[0,a]\cup Q_r[u_a,u_b]$. Note that $U[b,m+1]$ contains at least one reticulation vertex; otherwise $C'$ would be a cycle consisting only of tree vertices, contradicting Lemma~\ref{lmm:cycle}. Since $u_b\ne r$ lies on $Q_r$, the vertex $u_b$ is a tree vertex. Thus splitting $C$ at $u_b$ gives $k=\nret{C}=\nret{U[0,b]}+\nret{U[b,m+1]}$. As $b\ge j+1$ and $\nret{U[0,j+1]}=1$, we have $\nret{U[0,b]}\ge\nret{U[0,j+1]}=1$, hence $\nret{U[b,m+1]}\le k-1$. As $\nret{U[0,a]}=\nret{Q_r[u_a,u_b]}=0$, we have $\nret{C'}=\nret{U[b,m+1]}\le k-1$. Furthermore, as $v \in V(C')$, the induction hypothesis implies that there exists a reticulation vertex $r'$ on $C'$ such that $v$ is contained in the galled cycle $G_{r'}$ of $r'$.
As every reticulation vertex on $C'$ lies on $U[b,m+1]\subsetneq C$, it follows that $r'$ lies on $C$, as required.
\end{proof}

\begin{remark} \label{rmk:edgeambig}
Theorem~\ref{cycleandgalledcycle} guarantees the existence of a reticulation vertex $r$ on $C$ such that $v \in V(G_r)$, but does not specify which two of the edges of $\overline{N}$ incident to $v$ are contained in $G_r$.
\end{remark}

\begin{lemma} \label{lmm:treepathuni}
Let $N$ be a galled network, and let $u$ and $v$ be two distinct tree vertices in $N$. Then either every path between $u$ and $v$ contains the outgoing edge of some reticulation vertex, or there exists a unique tree path between $u$ and $v$ in $N$.
\end{lemma}

\begin{proof}
Since $N$ is connected, there exists at least one path between $u$ and $v$. Let $\mathcal{P}$ be the set of all such paths that avoid all reticulation vertex outgoing edges. If $\mathcal{P}=\emptyset$, then every path between $u$ and $v$ contains the outgoing edge of some reticulation vertex, and we are done. Hence assume that $\mathcal{P}\neq\emptyset$.

We show that $\mathcal{P}$ contains a tree path. Suppose not, such that every path in $\mathcal{P}$ contains at least one reticulation vertex. Let $P\in\mathcal{P}$ be a path with minimal number of reticulation vertices, and write $\nret{P}=k\geq 1$. 
Let $r$ be a reticulation vertex on $P$ with its corresponding galled cycle $G_r$. Since $P$ avoids the outgoing edge of $r$ and $r$ is not an endpoint of $P$, the path $P$ must contain the two incoming edges $\{p_1,r\}$ and $\{p_2,r\}$ of $r$. Hence $P$ contains the subpath $P_r=\{p_1,r\}\cup\{r,p_2\}$. Let $P_r^c$ be the tree path in $G_r$ between $p_1$ and $p_2$ that passes through the top vertex $s$ of $G_r$ and avoids $r$. Replacing $P_r$ in $P$ with $P_r^c$ yields a connected walk $W$ between $u$ and $v$ in $\overline{N}$. Hence $\nret{W}=k-1$. Note that the walk $W$ contains a path $P'$ between $u$ and $v$ consisting only of the vertices of $W$. Consequently, $P' \in \mathcal{P}$ and $\nret{P'} \leq k-1$, which contradicts the minimality of $k$. Therefore, there exists at least one tree path between $u$ and $v$ in $\mathcal{P}$. Furthermore,
by Lemma~\ref{lmm:cycle}, such a tree path is unique.
\end{proof}


\begin{corollary}\label{cor:treepath-2-connected}
Let $N$ be a galled network and let $H$ be a 2-connected subgraph of $\overline{N}$. Then there is a unique tree path in $\overline{N}$ between two distinct tree vertices $u,v\in V(H)$.
\end{corollary}

\begin{proof}
By Lemma~\ref{lmm:retcutvertex}, the outgoing edge of every reticulation vertex is a cut edge of $\overline{N}$, and hence lies on no cycle. Therefore $H$ contains no outgoing edge of a reticulation vertex.

Since $H$ is 2-connected, no path between $u$ and $v$ in $H$ contains the outgoing edge of a reticulation vertex, By Lemma~\ref{lmm:treepathuni} the unique tree path between $u$ and $v$ exists.
\end{proof}

\begin{lemma} \label{lmm:overlapone}
For any pair of distinct galled cycles in a galled network $N$, their intersection is either empty or a single tree path.
\end{lemma}

\begin{proof}
Since each reticulation vertex belongs to a unique galled cycle in $N$, two distinct galled cycles cannot share a reticulation vertex. Hence their overlap consists solely of tree vertices. Moreover, no maximal common connected subgraph can consist of a single vertex, since otherwise that vertex would be incident with four distinct edges in the two galled cycles, contradicting the degree constraints of a tree vertex.

Let $G_1$ and $G_2$ be two distinct galled cycles. For $i=1,2$, let $P_i$ be the tree path obtained from $G_i$ by deleting its reticulation vertex together with the two edges of $G_i$ incident with it. If $P_1\cap P_2$ were non-empty and not connected, then, by the fact that the union of two paths whose intersection is not connected contains a cycle, $P_1\cup P_2$ would contain a cycle whose vertices are all tree vertices, contradicting Lemma~\ref{lmm:cycle}. Hence $P_1\cap P_2$ is either empty or a single tree path.
\end{proof}

For two overlapping galled cycles $G_1$ and $G_2$, we refer to the single tree path on which they overlap as the \emph{intersecting path} of $G_1$ and $G_2$.

\begin{lemma} \label{lmm:twotopcompa}
Let $N$ be a galled network. If two distinct galled cycles $G_1$ and $G_2$ intersect, then their respective top vertices $s_1$ and $s_2$ satisfy either $s_1 \leq s_2$ or $s_2 \leq s_1$ in $N$.
\end{lemma}

\begin{proof}
Suppose for contradiction that $s_1$ and $s_2$ are incomparable such  that $s_1 \parallel s_2$. Let $v$ be the lowest common ancestor of $s_1$ and $s_2$ in $N$ such that $v \notin \{s_1, s_2\}$. Moreover, $v$ is neither the root nor a reticulation vertex, as the unique child of either would be a lower common ancestor of $s_1$ and $s_2$. Hence $v$ is a tree vertex, and by the minimality of $v$ there exist two directed paths $P_1$ from $v$ to $s_1$ and $P_2$ from $v$ to $s_2$ with $P_1 \cap P_2 = \{v\}$. Suppose, without loss of generality, that some reticulation vertex $r$ lies on $P_1$. Since $r\neq s_1$, the outgoing edge of $r$ lies on $P_1$. By Lemma~\ref{lmm:retcutvertex}, this edge is a cut edge of $\overline N$, and every path between $s_1$ and $s_2$ uses this edge. Note that a cut edge cannot lie on either of the cycles $G_1$ and $G_2$. If $G_1$ and $G_2$ intersect, then $G_1\cup G_2$ is connected, and hence contains a path between $s_1$ and $s_2$. This path avoids the cut edge, a contradiction.

Hence both $P_1$ and $P_2$ are directed tree paths.  We also write $P_i$ ($i=\{1,2\}$) for the corresponding undirected path in $\overline N$. By Lemma~\ref{lmm:overlapone}, $G_1$ and $G_2$ intersect on a single tree path $P$. By Lemma~\ref{lem:treepath-top}, $P$ has a unique highest vertex; denote it by $u$. Let $Q_i$ ($i=1,2$) be the directed subpath (and its corresponding undirected subpath) of $G_i$ from $s_i$ to $u$. Then $Q_1\cap Q_2=\{u\}$. Since $u$ is not a reticulation vertex, each $Q_i$ is a tree path.

Considering these paths in the underlying graph, $C := P_1 \cup P_2 \cup Q_1 \cup Q_2$ is a cycle. If $P_1$ and $Q_2$ in $N$ shared a vertex $w$, then $w\leq s_1$ and $s_2\leq w$, hence $s_2\leq s_1$, contradicting $s_1\parallel s_2$; the case of $P_2$ and $Q_1$ is symmetric. Note that $C$ contains only tree vertices, contradicting Lemma~\ref{lmm:cycle}, hence $s_1$ and $s_2$ are comparable.
\end{proof}

A tree vertex $v$ is called \emph{primitive} if there exist two distinct galled cycles $G_1$ and $G_2$ such that $v$ is an internal vertex on their intersecting path. In particular, $|V(G_1)\cap V(G_2)|\geq 3$.

\begin{lemma}\label{lmm:neighbourprimitive}
Let $v$ be a primitive tree vertex in a galled network $N$. For any two distinct galled cycles $G_1$ and $G_2$ such that $v$ is an internal vertex of their intersecting path, $v$ has a neighbour $u\notin V(G_1)\cup V(G_2)$.
\end{lemma}

\begin{proof}
Let $P=G_1\cap G_2$. Since $v$ is an internal vertex of $P$, two edges incident with $v$ lie on $P$. As $v$ is a tree vertex, $\deg_{\overline N}(v)=3$; let $u$ be its remaining neighbour in $\overline N$.

We claim that $u\notin V(G_1)\cup V(G_2)$. Suppose that $u\in V(G_1)$. Then $G_1 \cup \{u,v\}$ is 2-connected and both $u$ and $v$ have degree $3$ in it; hence, by Lemma~\ref{lmm:deg3-tree}, both are tree vertices. Since $G_1$ has exactly one reticulation vertex, one of the two subpaths of $G_1$ between $u$ and $v$ contains no reticulation vertex. Together with the edge ${u,v}$, they form a cycle consisting entirely of tree vertices, contradicting Lemma~\ref{lmm:cycle}. Thus $u\notin V(G_1)$. Similarly, $u\notin V(G_2)$, and hence $u\notin V(G_1)\cup V(G_2)$.
\end{proof}

\begin{lemma} \label{lmm:twotoploca}
Let $N$ be a galled network, and let $G_1,G_2$ be two distinct overlapping galled cycles in $N$ with top vertices $s_1,s_2$. Let $P$ be their intersecting path, with endpoints $t_1$ and $t_2$. Suppose, without loss of generality, that $s_2\leq s_1$. Then $s_2$ is the highest vertex of $P$.

In particular, exactly one of the following holds:
\begin{enumerate}
\item[(i)] $s_1=s_2$. In this case, $s_1$ is a primitive tree vertex, and $P[s_1,t_1]$ and $P[s_1,t_2]$ are two directed paths in $N$.
\item[(ii)] $s_2<s_1$. In this case, $P$ is a directed path in $N$ starting at $s_2$.
\end{enumerate}
\end{lemma}

\begin{proof}
Let $r_1$ and $r_2$ be the reticulation vertices of $G_1$ and $G_2$, respectively. First suppose that $s=s_1=s_2$. Since $s\in V(G_1)\cap V(G_2)$, we have $s\in V(P)$. As $s$ is the top vertex of both $G_1$ and $G_2$, it is also the highest vertex of $P$. We claim that $s$ is an internal vertex on $P$. Suppose that $s=t_1$. Then the common path $P$ uses one outgoing edge of $s$, while each of $G_1$ and $G_2$ has another outgoing edge from $s$ outside $P$. These three outgoing edges are distinct since $s=t_1$. This contradicts the fact that $s$ is a tree vertex. Hence $s$ is an internal vertex on $P$, so $s$ is primitive, and $P[s,t_1]$ and $P[s,t_2]$ are directed paths in $N$ by Lemma~\ref{lem:treepath-top}.

Now suppose that $s_2<s_1$. We first show that $s_2\in V(G_1)$. Suppose not. Let $Q$ be the directed path from $s_1$ to $s_2$, and let $u$ be the last vertex of $Q$ lying in $V(G_1)$. Then the subpath $Q[u,s_2]$ is internally disjoint from $G_1$ and contains no reticulation vertex; otherwise, an argument similar to the one in the proof of Lemma~\ref{lmm:twotopcompa} yields a contradiction to $G_1 \cap G_2 \neq \emptyset$.

Let $w$ be the highest vertex of $P$. Since $P\subseteq G_2$ and $s_2\notin V(G_1)$ by assumption, while $w\in V(G_1)$, we have $w<s_2<u$. Let $P_1$ be the subpath of $G_1$ between $u$ and $w$ that avoids $r_1$, and let $P_2$ be the directed subpath of $G_2$ from $s_2$ to $w$. Considering these paths in the underlying graph, $C:=P_1\cup P_2\cup Q[u,s_2]$ is a cycle consisting solely of tree vertices, contradicting Lemma~\ref{lmm:cycle}.

If $s_2$ were an internal vertex of $P$, then the two outgoing edges of $s_2$ in $G_2$ would both lie in $G_1$, forcing $s_2=s_1$, a contradiction. Hence $s_2$ is an endpoint of $P$ and the path $P$ is directed from $s_2$ by Lemma~\ref{lem:treepath-top}.
\end{proof}

\section{Planar Galled Networks} \label{sec:upward}

In this section, we prove Theorem~\ref{upward:thm}, which states that planarity and upward planarity are equivalent for galled networks. This equivalence was claimed without proof by Tollis and Kakoulis~\cite{TollisKakoulis2020}. It stands in contrast to the case of general phylogenetic networks, where the upward planar networks form a strict subset of the planar ones~\cite{MoultonWu2022}.

Given a planar drawing $\Gamma$ of a phylogenetic network $N$ and a path $P=(v_0,v_1,\dots,v_n)$ with $n\ge 1$, we write $\Gamma^\circ(P)=\Gamma(P)\setminus\{\Gamma(v_0),\Gamma(v_n)\}$ for its \emph{open arc}; in particular, for an edge $e=(u,v)$ this yields $\Gamma^\circ(e)=\Gamma^\circ(u,v)=\Gamma(e)\setminus\{\Gamma(u),\Gamma(v)\}$.
Then following Theorem~\ref{Phylogenetic Thomassen} on characterisation of upward planarity for phylogenetic networks is a special version of the well-known theorem by Thomassen~\cite[Theorem~5.1]{Thomassen1989} on single-source digraphs.

\begin{theorem}\label{Phylogenetic Thomassen}
Let $\Gamma$ be a planar drawing of a phylogenetic network $N$ with root $\rho$. Then there exists an upward planar drawing that is strongly equivalent to $\Gamma$ if and only if  $\Gamma(\rho)$ is contained on the boundary of the outer face of $\Gamma$, and every cycle $C$ in $N$ contains a reticulation vertex $r$ in $N$ with parents $p_1,p_2$ and child $v$ such that $\Gamma(p_1)$ and $\Gamma(p_2)$ are contained on the closed Jordan curve $\Gamma(C)$ while $\Gamma^\circ(r,v)$ is contained in $\mathrm{ext}(C)$.
Furthermore, if $N$ is a galled network, the node $\Gamma(v)$ is strictly contained in $\mathrm{ext}(C)$.
\end{theorem}

\begin{proof}

$(\Rightarrow)$ Suppose there exists an upward planar drawing of $N$ strongly equivalent to $\Gamma$, and let $C$ be a cycle of $N$. The root $\rho$ is the single source of $N$, so Thomassen~\cite[Thm.~5.1]{Thomassen1989} implies that $\Gamma(\rho)$ lies on the boundary of the outer face of $\Gamma$, and that $C$ has a vertex $r$ whose two neighbours on $C$ are both parents of $r$ and whose outgoing edges $(r,w)$ all satisfy $\Gamma^\circ(r,w)\not\subseteq\mathrm{Int}(C)$. Hence $r$ has indegree~2 and is a reticulation vertex, with parents $p_1,p_2\in V(C)$ and $\Gamma(p_1),\Gamma(p_2)\in\Gamma(C)$. Denote the unique child of $r$ by $v$ and we have $(r,v) \notin E(C)$. By planarity, $\Gamma^\circ(r,v)$ is disjoint from $\Gamma(C)$ and lies in a single face of $\mathbb{R}^2\setminus\Gamma(C)$. Since $\Gamma^\circ(r,v)\not\subseteq\mathrm{Int}(C)$, the face belongs to $\mathrm{ext}(C)$, i.e.\ $\Gamma^\circ(r,v)\subseteq\mathrm{ext}(C)$, as required.

If $N$ is a galled network, then $\{r,v\}$ is a cut edge of $\overline{N}$ by Lemma~\ref{lmm:retcutvertex}. If $v\in V(C)$, then $\{r,v\}\cup C$ is a $2$-connected graph, so by Lemma~\ref{lmm:deg3-tree} the vertex $r$ is a tree vertex, contradicting that $r$ is a reticulation. Hence $v\notin V(C)$ and $\Gamma(v)\notin\Gamma(C)$. As $\Gamma(v)$ is an endpoint of $\Gamma(r,v)$ with $\Gamma^\circ(r,v)\subseteq\mathrm{ext}(C)$, we have $\Gamma(v)\in\mathrm{Ext}(C)$; since moreover $\Gamma(v)\notin\Gamma(C)$, it follows that $\Gamma(v)\in\mathrm{ext}(C)$.

$(\Leftarrow)$ Conversely, let $C$ be a cycle of $N$ and let $r$ be the reticulation vertex provided by the hypothesis, with unique child $v$. Note that both of $r$'s neighbours on $C$ are its parents, and that the edge $(r, v)$ is its only outgoing edge.  By hypothesis $\Gamma^\circ(r,v)\subseteq\mathrm{ext}(C)$, and since $\mathrm{ext}(C)\cap\mathrm{Int}(C)=\emptyset$, it satisfies $\Gamma^\circ(r,v)\not\subseteq\mathrm{Int}(C)$. Thus $r$ meets both conditions of Thomassen~\cite[Thm.~5.1]{Thomassen1989}. As this holds for every cycle $C$ of $N$ and $\Gamma(\rho)$ lies on the boundary of the outer face, that theorem yields an upward planar drawing strongly equivalent to $\Gamma$.
\end{proof}

A local modification of a planar drawing relies on the concept of bridges~\cite[Section 9.4]{BoMu}. Note that this terminology is distinct from cut edges.

Given a cycle $C$ of $N$, a \emph{bridge} of cycle $C$ (or simply a bridge when $C$ is clear from context) is a maximal connected subgraph $B$ of $N$ such that no edge of $B$ belongs to $C$, and any two vertices of $B$ are connected by a path that is internally disjoint from $C$. The \emph{attachment vertices} of $B$ are the vertices in $V(B) \cap V(C)$. A bridge with $k$ vertices of attachment is a \emph{$k$-bridge}; for $k\ge 2$, these vertices partition $C$ into edge-disjoint paths called \emph{segments}.
Two bridges \emph{avoid} each other if all the attachment vertices of one bridge lie within a single segment partitioned by the attachment vertices of the other. Otherwise, they \emph{overlap}.

For a fixed planar drawing $\Gamma$, a bridge $B$ of $C$ is called an \emph{inner bridge} (resp. \emph{outer bridge}) in $\Gamma$ if $\Gamma(B)\setminus \Gamma(C)$ is contained in $\mathrm{int}(C)$ (resp. $\mathrm{ext}(C)$). When $\Gamma$ and $C$ are clear from the context, we simply call $B$ an inner (resp.\ outer) bridge. An inner bridge $B$ is \emph{transferable} if $B$ can be redrawn so that $\Gamma'(B)\setminus\Gamma(C)\subseteq \mathrm{ext}(C)$, while keeping the rest of the drawing fixed and preserving planarity. Bondy and Murty~\cite[Thm.~9.9]{BoMu} showed that an inner bridge avoiding every outer bridge is transferable. We obtain the following consequence

\begin{corollary}\label{cor:inner1bridge}
Every inner $1$-bridge is transferable.
\end{corollary}

\begin{proof}
A $1$-bridge has a single attachment vertex, which  lies in a single segment of any other bridge (or there is no other bridge). Hence it avoids every bridge, and is transferable.
\end{proof}

Figure~\ref{fig:bridge}(a) shows a planar drawing $\Gamma$ in which the cycle $C$ 
has five bridges: an inner $2$-bridge $B_1$ (thick black arcs inside $\Gamma(C)$), an outer $2$-bridge $B_2$ (dashed arc outside), an inner $4$-bridge $B_3$ (thin dashed arcs inside), an outer $2$-bridge $B_4$ (thick dashed arc outside), and an inner $1$-bridge $B_5$ (gray arcs inside). Here $B_1$, $B_2$, and $B_5$ avoid all other bridges, whereas $B_3$ and $B_4$ overlap. In particular, $B_1$ and $B_5$ avoid every outer bridge, so by Bondy and Murty~\cite[Thm.~9.9]{BoMu} both are transferable, while the inner $4$-bridge $B_3$ is not. Figure~\ref{fig:bridge}(b) shows the resulting planar drawing $\Gamma'$, in which $\Gamma'(B_1)$ and $\Gamma'(B_5)$ lie outside $\Gamma(C)$ while the rest of the drawing is unchanged.

\begin{figure}[htbp]
  \centering
  \includegraphics[width=0.96\textwidth]{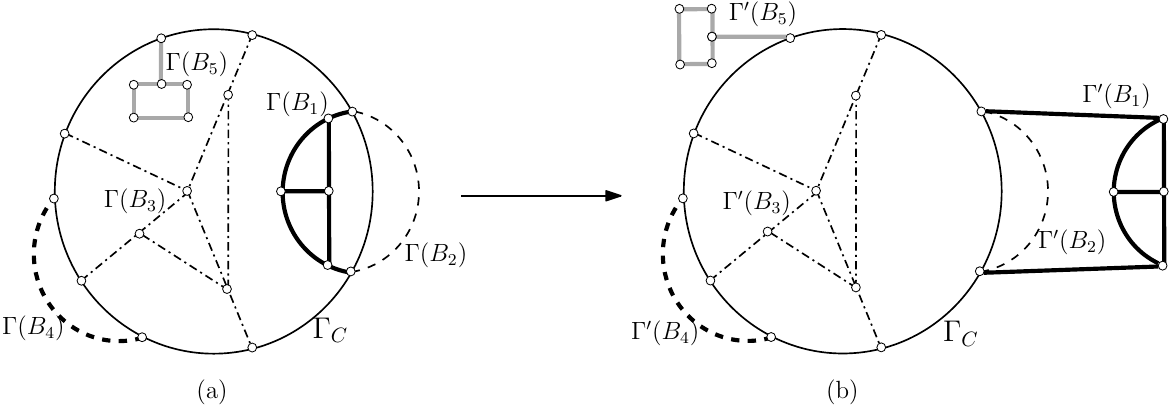}
  \caption{Bridges of a cycle $C$ and the transfer of inner bridges.
(a) A planar drawing $\Gamma$, with $\Gamma(C)$ shown as the large black circle.
(b) The drawing $\Gamma'$ obtained by transferring inner bridges $B_1$ and $B_5$ in $\Gamma$.}
  \label{fig:bridge}
\end{figure}

Fix a planar drawing $\Gamma$ of a graph $N$ and a cycle $C$ of $N$. Since the closed Jordan curve $\Gamma(C)$ partitions the plane, every bridge of $C$ is either inner or outer in $\Gamma$.
If $N$ is a galled network, a cycle $C$ of $N$, including a galled cycle, is called \emph{inward in $\Gamma$} if, for every reticulation vertex $r_i\in V(C)$, the unique child $v_i$ of $r_i$ satisfies $\Gamma(v_i)\in\mathrm{int}(C)$.
Furthermore, an inward cycle $C$ is called \emph{reducible in $\Gamma$} if there exists a reticulation vertex $r\in V(C)$ with galled cycle $G_r$ such that $G_r$ is inward in $\Gamma$; otherwise, $C$ is called \emph{irreducible in $\Gamma$}. In all these notions, we omit ``in $\Gamma$'' when $\Gamma$ is clear from the context.

\begin{lemma} \label{pinkcycleandgall}
Let $\Gamma$ be a planar drawing of a galled network $N$. Then every inward cycle in $\Gamma$ is reducible in $\Gamma$.
\end{lemma}

\begin{proof}
Suppose for contradiction that there exists an irreducible cycle $C$ in $\Gamma$. Let $W(C)$ denote the closed walk of $N$ whose drawing is the closed Jordan curve $\Gamma(C)$. Among all irreducible cycles, choose $C$ so that $W=W(C)$ has the minimum number of reticulation vertices.

Choose an arbitrary reticulation vertex $r$ in $V(W)$, and denote its galled cycle by $G_r$. Since $C$ is irreducible in $\Gamma$, $G_r$ is not inward in $\Gamma$, and hence $C\neq G_r$. Let the parents of $r$ be $p_1$ and $p_2$. As in the proof of Theorem~\ref{Phylogenetic Thomassen}, both $C$ and $G_r$ contain $\{p_1,r\}$ and $\{p_2,r\}$.

Write $W$ as $W=(v_1,v_2,\dots,v_m,v_1)$. Reindexing the label of vertices if necessary, let $P_0=(v_1,v_2,\dots,v_i)$ be the maximal subpath of $W$ containing $r$ whose vertices and edges are contained in $G_r$ such that $r \in V(P_0)$ and $P_0 \subseteq W\cap G_r$. Furthermore, $v_m \notin V(G_r)$, and let $u_1$ be the neighbour of $v_1$ in $G_r$ other than $v_2$; then $u_1 \notin V(W)$.  Similarly, $v_{i+1} \in V(W)$ but $v_{i+1} \notin V(G_r)$, while the neighbour of $v_i$ in $G_r$ other than $v_{i-1}$ does not appear in $W$.

Let $P'=(v_1,u_1,\dots,u_k)$ be the subpath of $G_r$ that is internally disjoint from $W$. That is, $v_1$ and $u_k$ are the only vertices of $V(P') \cap V(W)$. In particular, there exists an index $j$ with $i\le j\le m$ such that $v_j=u_k$.
Let $P_1=(v_1,v_2,\dots,v_j)$ be the subpath of $W$ between $v_1$ and $v_j$. Since $j\ge i$, $P_0$ is a subpath of $P_1$ (they are equal when $i=j$), and in particular $r\in V(P_1)$.
Let $P_2=(v_j,v_{j+1},\dots,v_m,v_1)$ be the complement subpath of $P_1$ in $W$. That is, $P_1$ and $P_2$ are internally disjoint, and their union is the cycle $C$, see Figure~\ref{reducible} as an example.

\begin{figure}[ht] 
    \centering
    \includegraphics[scale=0.5]{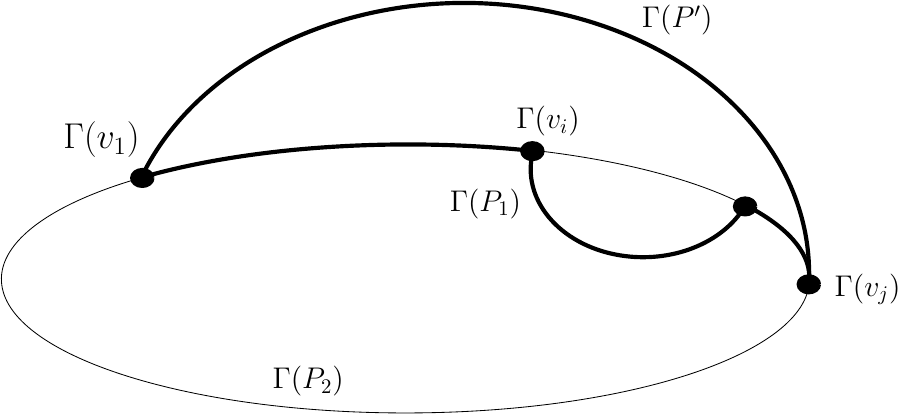}
\caption{A possible position of $\Gamma(P')$ and the closed Jordan curve $\Gamma(C) = \Gamma(P_1) \cup \Gamma(P_2)$. The bold closed Jordan curve is $\Gamma(G_r)$.}
\label{reducible}
\end{figure} 
Now consider the cycle $C_1$ (resp. $C_2$) formed by the union of the paths $P'$ and $P_1$ (resp. $P'$ and $P_2$). As $P'$ is a subpath of $G_r$ and $r \notin V(P')$, $P'$ is a tree path. Hence we have $\nret{P_1} \ge1 $ and $\nret{P_2} \ge1 $ otherwise $C_1$ or $C_2$ would be a cycle containing only tree vertices, contradicting Lemma~\ref{lmm:cycle}. Furthermore, for a reticulation vertex $r_t$ on $P_1$ or $P_2$ with child $w_t$ and galled cycle $G_t$, $G_t$ is not inward in $\Gamma$. Note that $v_1$ and $u_k=v_j$ are tree vertices by Lemma~\ref{lmm:deg3-tree} as they have degree~3 in $C \cup P'$; hence each $r_t$ is an internal vertex on $P_1$ or $P_2$.

By Lemma~\ref{lmm:diestel-412}, the Jordan curves $\Gamma(P_1)$, $\Gamma(P_2)$, and $\Gamma(P')$ partition the plane into three regions bounded by $\Gamma(C)=\Gamma(P_1)\cup\Gamma(P_2)$, $\Gamma(C_1)=\Gamma(P_1)\cup\Gamma(P')$, and $\Gamma(C_2)=\Gamma(P_2)\cup\Gamma(P')$. We have the following two cases regarding the position of $\Gamma(P')$, see Figure~\ref{fig:pink_cases}:

\begin{figure}[ht] 
    \centering
    \includegraphics[scale=0.7]{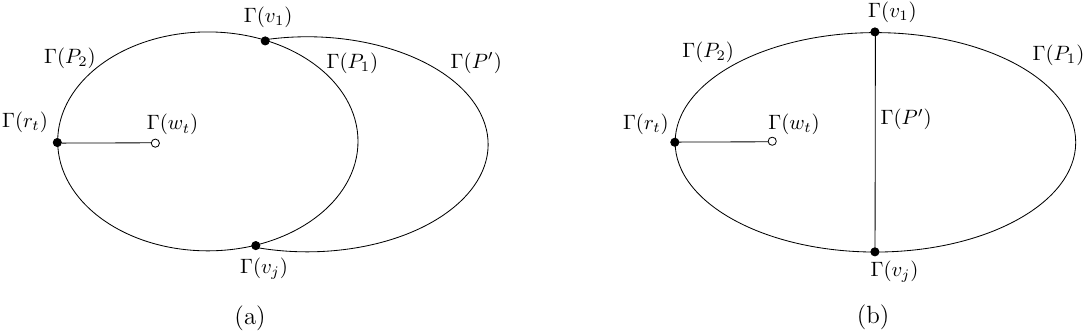}
\caption{(a) Case~1: $\Gamma^\circ(P')\subseteq\mathrm{ext}(C)$. (b) Case~2: $\Gamma^\circ(P')\subseteq\mathrm{int}(C)$. Here $r_t$ is a reticulation vertex on $P_1$ or $P_2$ with child $w_t$.}

\label{fig:pink_cases}
\end{figure}

\noindent\textit{Case 1: $\Gamma^\circ(P')\subseteq\mathrm{ext}(C)$.}

Since $\Gamma^\circ(P')\subseteq\mathrm{ext}(C)$, the region bounded by $\Gamma(C)$ is $\mathrm{int}(C)$,
while the other two regions are contained in $\mathrm{ext}(C)$, exactly one of them being unbounded. Let $C^\ast$ be the cycle in $\{C_1,C_2\}$ for which $\Gamma(C^\ast)$ is the boundary of the unbounded region. Since $P'$ is a tree path, every reticulation vertex on $C^*$ lies on $C$. Note that  $\mathrm{ext}(C^\ast)$ is disjoint from $\mathrm{int}(C)$, so  $\mathrm{int}(C)\subsetneq\mathrm{int}(C^\ast)$. Hence, as $C$ is inward, $C^\ast$ is inward. As $C$ is irreducible, the galled cycle of each reticulation vertex on $C$ is not inward; hence $C^\ast$ is irreducible. Finally, since both $P_1$ and $P_2$ contain a reticulation vertex, regardless of the choice of $C^\ast$, the part of $C$ not used by $C^\ast$ contains at least one reticulation vertex. Therefore $\nret{C^\ast} < \nret{C}$, contradicting the minimality of $C$. 

\smallskip

\noindent\textit{Case 2: $\Gamma^\circ(P')\subseteq\mathrm{int}(C)$.}

For a reticulation vertex $r_t$ on $P_2$, we claim that the drawing of its child $\Gamma(w_t)$ is contained in $\mathrm{int}(C_2)$. As in the proof of Theorem~\ref{Phylogenetic Thomassen}, we have
$w_t \notin V(C) \cup V(P')$. Suppose $\Gamma(w_t)\in \mathrm{int}(C_1)$; this contradicts the Jordan Curve Theorem, since $\Gamma(r_t)\in \mathrm{ext}(C_1)$ while $\Gamma(w_t)\in \mathrm{int}(C_1)$.  
Thus, every reticulation vertex $r_t$ on $C_2$ satisfies $\Gamma(w_t)\in\mathrm{int}(C_2)$, so $C_2$ is inward. Moreover, since none of the $G_t$ is inward, the cycle $C_2$ is irreducible. Since $C_2$ does not contain the reticulation vertex $r\in P_1$, we have $\nret{C_2} < \nret{C}$. This contradicts the minimality of $C$.
\end{proof}

Now we have the following characterisation for upward planar galled networks.

\begin{theorem}
\label{upward:thm}
A galled network is planar if and only if it is upward planar.
\end{theorem}

\begin{proof}
The implication from upward planarity to planarity is immediate. Conversely, let $\Gamma$ be a planar drawing of a galled network $N$, chosen so that $\Gamma(\rho)$ lies on the boundary of the outer face. For each reticulation vertex $r_i$, let $G_i$ be its galled cycle, $v_i$ its child, and let $Q$ be the number of inward galled cycles in $\Gamma$.

Suppose $Q>0$ and pick an inward galled cycle $G_i$. Let $B_i$ be the bridge of $G_i$ containing $v_i$. Since $\{r_i,v_i\}$ is a cut edge of $\overline N$ by Lemma~\ref{lmm:retcutvertex}, $B_i$ is an inner $1$-bridge of $G_i$ in $\Gamma$ with the unique attachment vertex $r_i$, hence transferable by Corollary~\ref{cor:inner1bridge}.

Let $F_i$ be the face in $\Gamma$ that is incident with $r_i$ but not incident with $v_i$. We construct a new planar drawing $\Gamma'$ from $\Gamma$ by transferring the inner bridge $B_i$ of $G_i$ in  $\Gamma$ such that $\Gamma'(B_i)$, the restriction of $\Gamma'$ to $B_i$, is strongly equivalent to $\Gamma(B_i)$, and $\Gamma'(B_i)$ is contained in the face $F_i$. Such a face $F_i$ is unique and satisfies $F_i\subseteq\mathrm{ext}(G_i)$, as $\{r_i,v_i\}$ is a cut edge and hence $\Gamma(r_i)$ is incident with only two faces. Since $B_i$ meets the rest of $\overline N$ only at $r_i$, placing a copy of $\Gamma(B_i-\{r_i\})$ in $F_i$ and redrawing $\Gamma(r_i,v_i)$ realises this while keeping the rest of $\Gamma$ fixed.

Now $\Gamma'(v_i)\in\mathrm{ext}(G_i)$, so $G_i$ is not inward in $\Gamma'$. For any other galled cycle $G_j$ we have $r_i\notin V(G_j)$, so $G_j$ is either disjoint from $B_i-\{r_i\}$ or contained in it. In the first case the drawing of $G_j$ is fixed. In the latter case, strong equivalence keeps the child of $G_j$ on the same side of $\Gamma(G_j)$. Hence no other inward status changes, $Q$ strictly decreases, and $\Gamma'(\rho)$ remains on the boundary of the outer face since $\Gamma$ and $\Gamma'$ differ only on $B_i$. Therefore, every vertex incident with the outer boundary of $\Gamma$ is also incident with the outer boundary of $\Gamma'$.

Repeating until $Q=0$ gives a drawing $\Gamma^*$ with no inward galled cycle, hence no inward cycle by Lemma~\ref{pinkcycleandgall}. Thus every cycle $C$ contains a reticulation vertex $r$ whose child $v$ satisfies $\Gamma^*(v)\notin\mathrm{int}(C)$. Since $\{r,v\}$ is a cut edge lying on no cycle, $v\notin V(C)$, and the two neighbours of $r$ on $C$ are its parents $p_1,p_2$, so $\Gamma^*(p_1),\Gamma^*(p_2)\in\Gamma^*(C)$. As $v\notin V(C)$, we have $\Gamma^*(v)\notin\Gamma^*(C)$. Combining this with 
$\Gamma^*(v)\notin\mathrm{int}(C)$, we get $\Gamma^*(v)\in\mathrm{ext}(C)$.  Since $\Gamma^*$ is planar, the arc $\Gamma^*(r,v)$ meets $\Gamma^*(C)$ only 
at $r$. As its other endpoint $\Gamma^*(v)$ lies in $\mathrm{ext}(C)$, the Jordan curve theorem gives $\Gamma^{*\circ}(r,v)\subseteq\mathrm{ext}(C)$.
The hypotheses of Theorem~\ref{Phylogenetic Thomassen} therefore hold, and $N$ admits an upward planar drawing.
\end{proof}

\section {Outerplanar Galled Networks} \label{sec:outer}

In this section we present three characterisations of outerplanar galled networks, each given by a different type of forbidden structure: forbidden vertex configurations, forbidden undirected subgraphs, and forbidden directed subgraphs.

A well-known theorem by 
Chartrand and Harary \cite{Chartrand-Harary} says  that an undirected graph is outerplanar if and only if it contains no subdivision of the complete graph $K_4$ or the complete bipartite graph $K_{2,3}$ as a subgraph which can be further strengthened 
for phylogenetic networks as follows.

\begin{lemma} \label{lmm:k4k23}
A phylogenetic network $N$ is outerplanar if and only if its underlying graph $\overline{N}$ contains no subdivision of $K_{2,3}$ as a subgraph.    
\end{lemma}
\begin{proof}
By Chartrand and Harary's characterisation of outerplanar graphs \cite{Chartrand-Harary}, 
it suffices to show that if $\overline{N}$ contains a subdivision of $K_4$, then $\overline{N}$ contains a subdivision of $K_{2,3}$.

Suppose that $\overline{N}$ contains a $K_4$. Then, in $N$, each edge of this $K_4$ is directed. Let its vertices be $v_1, v_2, v_3, v_4$, where each $v_i$ is a tree vertex by Lemma~\ref{lmm:deg3-tree}. Suppose that $N$ contains the edges $(v_2,v_1)$, $(v_1,v_3)$, and $(v_1,v_4)$. However, either direction of the edge $\{v_3,v_4\}$ contradicts the definition of a tree vertex for $v_3$ and $v_4$, since in either case one of the two vertices would be a reticulation vertex. Hence $\overline{N}$ contains no subgraph isomorphic to $K_4$.

Furthermore, we claim that any subdivision of $K_4$ contains a subdivision of $K_{2,3}$ as a subgraph. Let $H$ be a subdivision of $K_4$ with branch vertices $v_1, v_2, v_3, v_4$ and canonical paths $P_{i,j}$ corresponding to the edges $\{v_i, v_j\}$ of $K_4$. Choose a degree~$2$ vertex $s$ on some canonical path $P_{i,j}$. Then $H$ contains a subdivision of $K_{2,3}$ with degree~$3$ branch vertices $\{v_i, v_j\}$ and degree~$2$ branch vertices $\{v_k, v_l, s\}$. Its canonical paths are $P_{i,k}, P_{i,l}, P_{j,k}, P_{j,l}$, together with the two subpaths of $P_{i,j}$ obtained by splitting it at $s$.
\end{proof}

In view of Lemma~\ref{lmm:k4k23}, we could focus on 
subdivisions $H$ of $K_{2,3}$ for phylogenetic networks.
For later use, we say that $H$, or equivalently the corresponding subdivision of $K_{2,3}$ in the underlying graph of a phylogenetic network, is \emph{normal} if one of the three canonical paths between the two degree~$3$ branch vertices $t_1$ and $t_2$ of $H$ is a tree path. Note that by Lemma~\ref{lmm:cycle}, this is equivalent to saying that precisely one of the three canonical paths between the two degree~$3$ branch vertices $t_1$ and $t_2$ of $H$ is a tree path.
Otherwise, we call it \emph{not normal} (See Figure~\ref{k23andnormal} for an example). When no confusion can arise, we also say  that a subdivision of $K_{2,3}$ is normal or not normal. This enables us to provide a stronger version of Lemma~\ref{lmm:k4k23} for galled networks.  

\begin{figure}[h]
    \centering
    \includegraphics[scale=0.6]{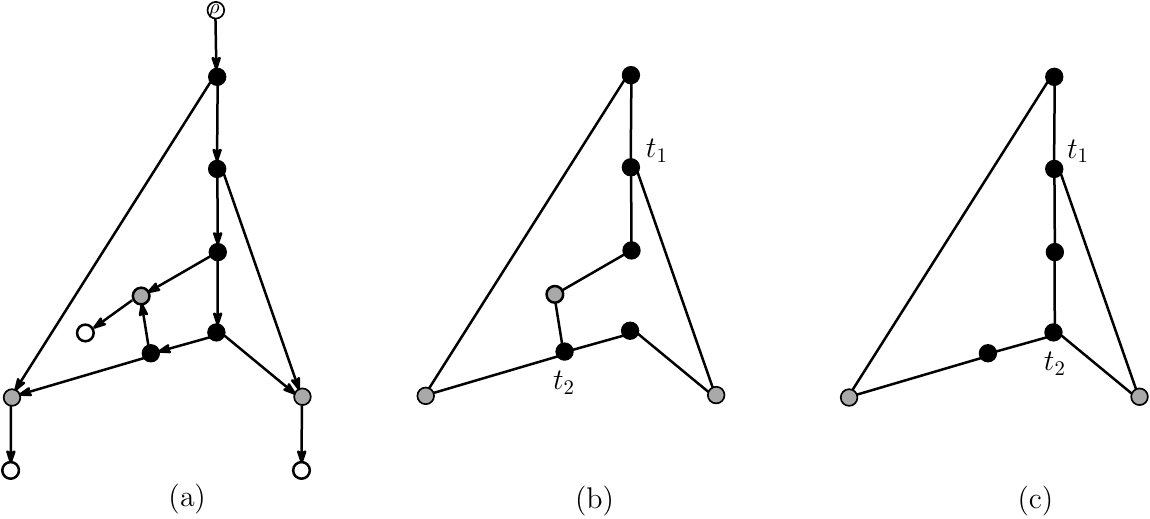}
    \caption{(a) A non-outerplanar galled network $N$. (b) A subdivision of $K_{2,3}$ in $\overline{N}$: each of its three paths between the two degree~$3$ branch vertices contains a reticulation vertex, so it is not normal. (c) Another subdivision of $K_{2,3}$ in $\overline{N}$: exactly one of the three paths between the two degree~$3$ branch vertices is a tree path, so it is normal.}
    \label{k23andnormal}
\end{figure}

\begin{theorem}\label{thm:normalk23}
A galled network $N$ is  outerplanar if and only if its underlying graph $\overline{N}$ does not contain a normal subdivision of  $K_{2,3}$ as a subgraph. 
\end{theorem}

\begin{proof}
By Lemma~\ref{lmm:k4k23}, it is sufficient to prove that if $\overline{N}$ contains a subdivision $H$ of $K_{2,3}$ as a subgraph, then  $\overline{N}$ contains a normal subdivision of $K_{2,3}$.

Denote the two  degree~$3$ branch vertices of $H$ by $t_1$ and $t_2$, and the three internally disjoint paths between them by $P_1,P_2,P_3$. Since $H$ is 2-connected,  $t_1$ and $t_2$ are tree vertices in view of  Lemma~\ref{lmm:deg3-tree} as $\deg_H(t_1)=\deg_H(t_2)=3$. Furthermore, we may assume that $H$ is not normal, as otherwise the theorem clearly holds. This implies that 
for $1\le i \le 3$, the path $P_i$ is not a tree path, and hence contains at least one reticulation vertex.

On the other hand, as $H$ is 2-connected, Corollary~\ref{cor:treepath-2-connected} implies that there exists a unique tree path $P=(w_0:=t_1,w_1,\dots,w_n:=t_2)$ between $t_1$ and $t_2$ in $\overline{N}$.
Denote the three neighbours of  $t_1$ in $\overline{N}$ by $v_1,v_2,v_3$ such that $v_i$ is contained in $P_i$ for $1\le i \le 3$. 
Swapping the index of $P_i$ if necessary, we may further assume that $w_1=v_1$.

Note that if $P$ is internally disjoint from $P_2$ and $P_3$, then 
$P$, $P_2$, and $P_3$ are three internally disjoint paths between $t_1$ and $t_2$. 
This shows that $P\cup P_2\cup P_3$ is a normal subdivision of $K_{2,3}$, as required. 

Therefore it remains to consider the case that $P$ has an internal vertex on $P_2\cup P_3$.
Let $j$ be the smallest index with $0<j<n$ such that $w_j\in V(P_2)\cup V(P_3)$, and set $u:=w_j$. Without loss of generality, assume that $u\in V(P_3)$. 
By the minimality of $j$, $Q_1 := P[t_1,u]$, $Q_2 := P_2\cup P_3[t_2,u]$, and $Q_3 := P_3[t_1,u]$ are three pairwise internally disjoint paths between $t_1$ and $u$.  
Furthermore, as a subpath of $P$, $Q_1$ is a tree path.
Hence $H':=Q_1\cup Q_2\cup Q_3$ is a normal subdivision of $K_{2,3}$ with branch vertices $t_1$ and $u$, completing the proof.
\end{proof}

Using Theorem~\ref{thm:normalk23}, we will provide another characterisation of outerplanar galled networks using two types of forbidden vertex configurations. The first type is given by primitive tree vertices introduced in Section~\ref{sec:galled}, namely internal tree vertices on the intersecting paths of two galled cycles.
The second type is is given by {\em separable} tree vertices, that is, a tree vertex $v$ contained in precisely three galled cycles $G_1,G_2,G_3$ such that for the three neighbours $v_1,v_2,v_3$ of $v$, we have $v_i\not \in V(G_i)$ for each $1\le i \le 3$.  Note that a vertex cannot be simultaneously separable and primitive. Furthermore, we have the following two properties concerning separable and primitive tree vertices.

\begin{lemma} \label{lmm:sepuni}
If $v$ is a separable tree vertex, then $v$ is the unique common intersection of the three galled cycles containing $v$.
\end{lemma}

\begin{proof}
Since $\{v, v_i\} \notin E(G_i)$ for $1 \le i \le 3$, the pairwise intersection $G_j \cap G_k$ contains the edge $\{v, v_i\}$ at $v$, so the three pairwise intersecting paths contain distinct edges at $v$. Suppose for contradiction that there exists $w \in V(G_1) \cap V(G_2) \cap V(G_3)$ with $w \neq v$. By Lemma~\ref{lmm:overlapone}, each pairwise intersection of galled cycles is a single tree path, and the subpaths of $G_1 \cap G_2$ and $G_2 \cap G_3$ between $v$ and $w$ are two tree paths containing distinct edges at $v$. Their union contains a cycle consisting entirely of tree vertices, contradicting Lemma~\ref{lmm:cycle}.
\end{proof}

\begin{lemma}\label{lmm:three-galled-cycles}
Let $v$ be a tree vertex in a galled network $N$ such that the number of galled cycles containing $v$ is at least three. 
Then $v$ is either primitive or separable (but not both).
\end{lemma}

\begin{proof}
Let $G_1,G_2,\cdots,G_k$ be the set of galled cycles containing $v$ for some $k\ge 3$. 
Let $v_1,v_2,v_3$ be the neighbours of $v$ in $N$. 
Consider the bipartite graph $H$ with bipartition $\{v_1,v_2,v_3\}$ and $\{G_1,\dots,G_k\}$, where each galled cycle $G_j$ is regarded as a single vertex of $H$, and we put an edge $\{v_i,G_j\}$ in $H$ if and only if $v_i\notin V(G_j)$.
Since each $G_i$ contains precisely two of the three neighbours of $v$, it follows that each $G_j$ is connected to precisely one $v_i$. Now we have two cases to consider. 

In the first case, each $v_i$ has degree one in $H$.  Then we have $k=3$. Relabelling their indices if necessary, we may assume that $\{v_i,G_i\}$ is an edge in $H$ for $1\le i \le 3$. This implies $v_i\not \in V(G_i)$ for each $i$ and hence $v$ is separable.
Furthermore, for $1\le i <j\le 3$, neither $v_i$ nor $v_j$ is contained in $V(G_i)\cap V(G_j)$. This implies that $v$ is an endpoint of $G_i\cap G_j$, and hence $v$ is not primitive.

The other case is that some $v_i$ has degree at least two in $H$. Without loss of generality, we may assume that $v_1$ is connected to at least two galled cycles, say $G_1$ and $G_2$. This implies that neither $G_1$ nor $G_2$ contains $v_1$ in $N$, and hence both galled cycles contain the edges $\{v,v_2\}$ and $\{v,v_3\}$, so $v$ is an internal vertex on the intersecting path of $G_1$ and $G_2$; by definition, $v$ is primitive. 
To prove that $v$ is not separable, it suffices to consider the case that $k=3$. Since $G_3$ is connected to only one $v_i$ in $H$, by swapping $v_2$ and $v_3$ if necessary, we may assume that $v_3$ has  degree $0$ in $H$. This implies $v_3\in V(G_1)\cap V(G_2)\cap V(G_3)$ with $v_3\ne v$.
If $v$ were separable, Lemma~\ref{lmm:sepuni} would force $v$ to be the unique common vertex of $G_1,G_2,G_3$, a contradiction. Hence $v$ is not separable.
\end{proof}

Now we are in a position to present the  characterisation of outerplanar galled networks using forbidden vertex configurations.

\begin{theorem}\label{thm:outerplanar-forbidden}
A galled network $N$ is outerplanar if and only if it contains neither a  primitive tree vertex nor a separable tree vertex.
\end{theorem}

\begin{proof}
By Theorem~\ref{thm:normalk23}, $N$ is outerplanar if and only if $\overline{N}$ contains no normal subdivision of $K_{2,3}$. Hence it suffices to show that $\overline{N}$ contains a normal subdivision of $K_{2,3}$ if and only if $N$ contains a separable or primitive tree vertex.

($\Leftarrow$)
Suppose first that $N$ contains a primitive tree vertex $v$. Then $v$ is an internal tree vertex on the intersecting path $Q$ between two galled cycles $G_1$ and $G_2$, where $Q$ has endpoints $t_1$ and $t_2$. In particular, $|V(Q)|\geq 3$. The non-shared paths of $G_1$ and $G_2$ are internally disjoint from $Q$ and from each other. Furthermore, for the reticulation vertex $r_1$ of $G_1$ and $r_2$ of $G_2$, we have $r_1 \notin V(Q)$ and $r_2 \notin V(Q)$. Hence $Q$ together with the two non-shared paths of $G_1$ and $G_2$ forms a normal subdivision of $K_{2,3}$ in $\overline{N}$ with degree~$3$ branch vertices $t_1$ and $t_2$ and degree~$2$ branch vertices $r_1$, $r_2$, and $v$.

Suppose next that $N$ contains a separable tree vertex $v$ with neighbours $u_1,u_2,u_3$. Then there exist precisely three galled cycles $G_1,G_2$, and $G_3$ containing $v$ such that $u_i \notin V(G_i)$. Let $e_i=\{v,u_i\}$ for $1 \le i \le 3$. This implies that $e_i \in G_j\cap G_k$ for $\{i,j,k\}=\{1,2,3\}$. By Lemma~\ref{lmm:overlapone}, $G_j\cap G_k$ is a tree path containing $e_i$ and having $v$ as one endpoint; denote its other endpoint by $w_i$. By Lemma~\ref{lmm:sepuni}, $V(G_1)\cap V(G_2)\cap V(G_3)=\{v\}$, so $w_1,w_2,w_3$ are pairwise distinct. Furthermore, since $G_i$ avoids $e_i$, the two shared tree paths of $G_i$ end at $w_j$ and $w_k$, and hence the non-shared path of $G_i$ connects $w_j$ with $w_k$. Hence $G_1\cup G_2\cup G_3$ contains a subdivision of $K_4$ with branch vertices $v,w_1,w_2,w_3$, and by Theorem~\ref{thm:normalk23}, $\overline{N}$ contains a normal subdivision of $K_{2,3}$.

\smallskip

($\Rightarrow$)
Let $H$ be a normal subdivision of $K_{2,3}$ in $\overline{N}$, with degree $3$ branch vertices $t_1,t_2$ and three internally disjoint nontrivial paths $P_1,P_2,P_3$ between $t_1$ and $t_2$. We assume that $P_1$ is the unique tree path between $t_1$ and $t_2$ in $H$. Choose an internal vertex $v$ of $P_1$. For $i=1,2$, let $P_1^i$ denote the subpath $P_1[t_i,v]$, and let $u_i$ be the neighbour of $v$ contained in $P_1^i$. Let $u_3$ be the other neighbour of $v$ that is distinct from $u_1$ and $u_2$. Set $e_i:=\{v,u_i\}$ for $i\in\{1,2,3\}$. Consider the two cycles $P_1\cup P_2$ and $P_1\cup P_3$. Since $P_1$ is a tree path, applying Theorem~\ref{cycleandgalledcycle} to each cycle yields reticulation vertices $r_1$ and $r_2$, which are internal vertices of $P_2$ and $P_3$, respectively, whose galled cycles $G_1$ and $G_2$ both contain $v$. Since $P_2$ and $P_3$ are internally disjoint, we have $r_1\neq r_2$.

If $e_3$ is a cut edge of $\overline{N}$, then $e_3 \notin E(G_1)$ and $e_3 \notin E(G_2)$. Since $G_1$ and $G_2$ are cycles containing $v$ but neither can contain $e_3$, both must contain the edges $e_1$ and $e_2$. Therefore, $v$ is an internal vertex on the intersecting path of $G_1$ and $G_2$. By definition, $v$ is a primitive tree vertex.

Now we may assume that $e_3$ is not a cut edge in $\overline{N}$. Hence $\overline{N}-v$ is connected, and there is a vertex $w\in V(H)\setminus\{t_1,t_2\}$ with $w\ne v$ and a path $P$ between $v$ and   $w$ containing $e_3$ and internally disjoint from $H$. Since $P$ is internally disjoint from $H$ and has both endpoints in $H$, the graph $H\cup P$ is also $2$-connected. As $\deg_{H \cup P}(w) = 3$, by Lemma~\ref{lmm:deg3-tree}, $w$ is a tree vertex of $N$. Then we have the following two cases regarding the positions of $w$.

\smallskip
\noindent\emph{(a) $w$ is contained in $P_1$.}
In this case, $P\cup P_1[v,w]$ is a cycle containing $v$.
By Theorem~\ref{cycleandgalledcycle} and the fact that $P_1[v,w]$ is a tree path, there exists a reticulation vertex $r_3$ which is an internal vertex on $P$ such that $v \in V(G_3)$. Since $P$ is internally disjoint from $H$ while $r_1 \in V(P_2)$ and $r_2 \in V(P_3)$, we have $r_3 \notin \{r_1, r_2\}$. It follows that $v$ is contained in at least three distinct galled cycles. By Lemma~\ref{lmm:three-galled-cycles}, $v$ is either primitive or separable.

\smallskip
\noindent\emph{(b) $w$ is not contained in $P_1$.} Reindexing if necessary, we may assume that $w\in V(P_2)$. 
Then $w$ splits $P_2$ into two subpaths $P_2[t_1,w]$ and $P_2[w,t_2]$.
As $P$ is internally disjoint from $H$ and the paths $P_1,P_2,P_3$ are pairwise internally disjoint, we have two cycles $C_1:=P\cup P_2[t_1,w]\cup P_1^1$ and $C_2:=P\cup P_2[w,t_2]\cup P_1^2$. Since $w$ is a tree vertex, $r_1\ne w$, so $r_1$ lies on exactly one of the two subpaths of $P_2$. Without loss of generality, assume that $r_1\in V(P_2[t_1,w])$. Now consider the cycle $C_2$, then $r_1\notin V(C_2)$. Since $C_2$ is internally disjoint from $P_3$, we also have $r_2\notin V(C_2)$. Applying Theorem~\ref{cycleandgalledcycle} to $C_2$, there is a reticulation vertex $r_3\in V(C_2)$ such that $v\in V(G_3)$. Since $r_1 \notin V(C_2)$ and $r_2 \notin V(C_2)$, we have $r_3 \notin \{r_1, r_2\}$. It follows that $v$ is contained in at least three distinct galled cycles. By Lemma~\ref{lmm:three-galled-cycles}, $v$ is either primitive or separable. This completes the proof.
\end{proof}

\begin{remark}
Lemma~\ref{lmm:three-galled-cycles} and Theorem~\ref{thm:outerplanar-forbidden} imply that in an outerplanar galled network $N$, no vertex can be contained in three or more distinct galled cycles.
\end{remark}

In the remainder of this section, we will present the final characterisation of outerplanar galled networks using forbidden directed graphs. To this end, we recall the definition of cut-visible vertices introduced in Section~\ref{sec:background}.
Let $H$ be a subgraph of a graph $G$. A vertex $v \in V(H)$ is said to be \emph{cut-visible} from $H$ in $G$ if there exists a path $P$ in $G$ connecting $v$ to a cut vertex of $G$ such that $V(P) \cap V(H) = \{v\}$ and $P$ contains exactly one cut vertex of $G$, which is an endpoint of $P$. In particular, if $v$ itself is a cut vertex of $G$, then $P$ can be taken to be the degenerate path consisting solely of $v$.

\begin{lemma}\label{lmm:overlap-path-property}
Let $N$ be a galled network, and let $v$ be a primitive tree vertex that is an internal vertex of the intersecting path $P$ of two galled cycles $G_1$ and $G_2$. Let $H := G_1 \cup G_2$. Then we have the following three assertions:

(i) For any vertex $w \in V(H)$, every path in $N$ between $v$ and $w$ that is internally disjoint from $H$ has at least one reticulation vertex as an internal vertex.

(ii) $v$ is cut-visible from $H$ in $\overline{N}$.

(iii) There exists a path $P^*$ between $v$ and $x$ for $x\in \{\rho\}\cup X$ such that $P^*$ is internally disjoint from $H$.

\end{lemma}

\begin{proof}
Let $t_1, t_2$ be the endpoints of the shared tree path $P$. For $i=1,2$, let $v_i$ be the neighbour of $v$ belonging to the subpath $P_i := P[v,t_i]$. Denote the third neighbour of $v$ by $v_3$, and let $e := \{v,v_3\}$. By Lemma~\ref{lmm:neighbourprimitive}, we have $v_3\notin V(H)$. As $v_1, v_2$ lie on the galled cycles $G_1, G_2$, neither $\{v, v_1\}$ nor $\{v, v_2\}$ is a cut edge in $\overline{N}$; hence $v$ is a cut vertex of $\overline{N}$ if and only if $e$ is a cut edge.

\smallskip
\noindent\textit{Case 1: $v$ is a cut vertex of $\overline{N}$.}

In this case, no path in $\overline{N}$ between $v$ and a vertex of $V(H)$ can be internally disjoint from $H$, so (i) is satisfied. Moreover, (ii) is satisfied since $v$ itself is a cut vertex. It remains to prove (iii). Let $B$ denote the maximal connected subgraph of $\overline{N} - v$ containing $v_3$; since $v$ is a cut vertex, $B$ is disjoint from $V(H) \setminus \{v\}$. If $e=(v,v_3)$, following outgoing edges in $N$ from $v$ along $e$ yields a directed path from $v$ to a leaf; if $e=(v_3,v)$, following incoming edges from $v$ along $e$ yields a directed path from the root $\rho$ to $v$. Either way, all internal vertices of the path lie in $B$, so the path is internally disjoint from $H$.

\smallskip
\noindent\textit{Case 2: $v$ is not a cut vertex of $\overline{N}$.}

In this case, there exists a path $P'$ between $v$ and a vertex $w \in V(H)$ that is internally disjoint from $H$ and contains $e$. For (i), let $r_1$ denote the unique reticulation vertex of $G_1$ and assume that $P'$ is a tree path. Without loss of generality, assume $w \in V(G_1)$. Note that $w \ne t_i$ for $i \in \{1,2\}$, since otherwise $\deg_{\overline{N}}(t_i)=4$. Furthermore, $w\neq r_1$ as $\deg_{G_1 \cup P'}(w)=3$ hence $w$ is a tree vertex. Therefore, $v$ and $w$ divide $G_1$ into two internally disjoint paths: $Q'_1$ containing $r_1$, and $Q'_2$ containing only tree vertices. Since $P'$ is internally disjoint from $G_1$, the paths $Q'_2$ and $P'$ are internally disjoint, and the cycle $Q'_2 \cup P'$ consists entirely of tree vertices, contradicting Lemma~\ref{lmm:cycle}. Hence $P'$ contains at least one reticulation vertex.

For (ii), $P'$ contains a reticulation vertex by (i), which is a cut vertex of $\overline{N}$ by Lemma~\ref{lmm:retcutvertex}. Hence $P'$ contains a one cut vertex. Let $v_c$ be the cut vertex of $\overline{N}$ on $P'$ closest to $v$. Then $V(P'[v,v_c]) \cap V(H) = \{v\}$ and $v_c$ is the only cut vertex of $P'[v,v_c]$, so $v$ is cut-visible from $H$ in $\overline{N}$.

For (iii), let $r$ be the reticulation vertex of $P'$ closest to $v$. Since $V(P'[v,r]) \cap V(H)=\{v\}$ and $r$ is a cut vertex, the subnetwork below $r$ shares no common vertex with $H$. Extending along the outgoing edge from $r$ yields a directed path $P_{\ell}$ from $r$ to a leaf, vertex-disjoint from $H$. Viewing $P_{\ell}$ as a path in $\overline N$, the union $P^*:=P'[v,r]\cup P_{\ell}$ is a path between $v$ and a leaf, internally disjoint from $H$.
\end{proof}

\begin{remark}\label{rmk:to-root}
In the proof of Lemma~\ref{lmm:overlap-path-property} Case 2(iii), there may also exist a path between the root $\rho$ and $v$ that is internally disjoint from $H$. Specifically, if some tree vertex $t$ on the subpath $P'[v, r]$ has an incoming edge not lying on $P'$ whose reverse extension to $\rho$ is internally disjoint from $H$, then combining this extension path with $P'[v, t]$ yields the required path in $\overline{N}$.
\end{remark}

\begin{theorem}\label{outer forbidden}
A galled network $N$ is outerplanar if and only if $N$ contains no directed subdivision of any of $D_1,D_2,D_3$ in Figure~\ref{fig:mainforbidden} as a subgraph.
\end{theorem}

\begin{proof}

($\Rightarrow$) Suppose $N$ contains a directed subdivision of $D_j$ for some $j \in \{1,2,3\}$. Note that $D_1$ and $D_2$ each contain a primitive tree vertex, and $D_3$ contains a separable tree vertex. Hence $N$ is not outerplanar by Theorem~\ref{thm:outerplanar-forbidden}.

\smallskip

($\Leftarrow$) Suppose $N$ is not outerplanar. Then by Theorem~\ref{thm:outerplanar-forbidden}, we have the following two cases.

\smallskip

\noindent\textit{Case (i). $N$ contains a primitive tree vertex $v$.} 

Then there exist two distinct galled cycles $G_1$ and $G_2$ such that $v$ is an internal vertex on the intersecting path $P$, whose two endpoints are denoted by $t_1, t_2$.  For $i \in \{1, 2\}$, let $r_i$ and $s_i$ be the reticulation vertex and top vertex of $G_i$, respectively, and let $\ell_i$ be a leaf reached from $r_i$ along a directed path. By Lemma~\ref{lmm:twotopcompa}, swapping the indices  if needed, we have either $s_1 = s_2$ or $s_1 > s_2$.

First consider the subcase where $s_1=s_2$. By Lemma~\ref{lmm:twotoploca}(i), $s:=s_1$ is an internal vertex on $P$ and also the highest vertex of $P$, and $P$ consists of two directed subpaths from $s$ to $t_1$ and from $s$ to $t_2$.
Therefore, both $t_1$ and $t_2$ are lowest common ancestors of $r_1$ and $r_2$. There exists a directed path $P_{i,j}$ from $t_i$ to $r_j$ along $G_j$ for $1\le i,j\le 2$ such that the four paths $P_{i,j}$ are pairwise internally disjoint directed paths. Combined with the directed path from $\rho$ to $s$ that is internally disjoint from $H$ and the directed paths from $r_1$ to $\ell_1$ and from $r_2$ to $\ell_2$, $N$ contains a directed subdivision of $D_1$.

Next consider the other subcase $s_1 > s_2$.
By Lemma~\ref{lmm:twotoploca}(ii), $s_2$ is the highest vertex of $P$, relabelling $t_1,t_2$ if necessary, we may assume $s_2 = t_1$. $P$ is a directed path from $s_2$ to the other endpoint $t_2$.
Since $s_1 > s_2$ and $s_2 \in V(G_1)$, the galled cycle $G_1$ contains a directed path from $s_1$ to $s_2$, and $t_2$ is the lowest common ancestor of $r_1$ and $r_2$ with two internally disjoint directed paths from $t_2$ to $r_1$ and from $t_2$ to $r_2$ along $G_1$ and $G_2$, respectively.
Combined with the directed path from $\rho$ to $s_1$, the other directed path from $s_1$ to $r_1$ along $G_1$ (not passing through $s_2$ and $t_2$), the directed path from $s_2$ to $r_2$ along $G_2$ (not passing through $t_2$), and the directed paths from $r_1$ to $\ell_1$ and from $r_2$ to $\ell_2$, $N$ contains a directed subdivision of $D_2$.

\smallskip
\noindent\textit{Case (ii). $N$ contains a separable tree vertex $v$ but contains no primitive tree vertex.}

By definition, $v$ is contained in exactly three distinct galled cycles $G_1, G_2, G_3$, with neighbours $v_1, v_2, v_3$ such that $v_i \notin V(G_i)$ for each $i \in \{1,2,3\}$. Let $H = G_1 \cup G_2 \cup G_3$. For $i \in \{1,2,3\}$, let $r_i, s_i$ denote the reticulation vertex and top vertex of $G_i$, and let $\ell_i$ be the leaf reached from $r_i$ along the directed path internally disjoint from $H$.

For $1\le i<j\le 3$, let $\{k\}=\{1,2,3\}\setminus\{i,j\}$ and consider the common path $P_{i,j}=G_i\cap G_j$. 
Since $v$ is separable, $P_{i,j}$ contains $v$ as an endpoint as shown in the proof of Lemma~\ref{lmm:three-galled-cycles}. Furthermore, as $N$ does not contain a primitive tree vertex, it follows that $P_{i,j}$ is an edge, which is necessarily $\{v,v_k\}$.
Furthermore, by Lemma~\ref{lmm:twotopcompa} and Lemma~\ref{lmm:twotoploca}, we have either $s_i<s_j$ or $s_j<s_i$ since $N$ does not contain any primitive tree vertex. Reindexing if necessary, assume $s_1 < s_2 < s_3$. By Lemma~\ref{lmm:twotoploca}(ii), $s_1$ lies on all three cycles, that is, $s_1 \in V(G_1) \cap V(G_2) \cap V(G_3)$. By the uniqueness in Lemma~\ref{lmm:sepuni} we have $s_1 = v$. As $s_1 \neq s_2$, since otherwise $N$ contains a primitive tree vertex, for $P_{2,3}=\{v,v_1\}$, we have $s_2=v_1$. Hence the directions of the edges incident with $v$ are $(v_1,v)$, $(v,v_2)$, and $(v,v_3)$.

Since $s_3$ is the unique top vertex of $H$, there is a directed path from $\rho$ to $s_3$ internally disjoint from $H$.
Combined with the two directed paths from $s_3$ to $r_3$ along $G_3$ (one passing through $v_1$, $v$, $v_2$, the other not), the two directed paths from $s_2$ to $r_2$ along $G_2$ (one passing through $v$, $v_3$, the other not), the two directed paths from $s_1$ to $r_1$ along $G_1$ (one passing through $v_2$, the other through $v_3$), and the directed paths from $r_1, r_2, r_3$ to $\ell_1, \ell_2, \ell_3$, $N$ contains a directed subdivision of $D_3$.
\end{proof}

\section{Terminal planar galled networks} \label{sec:terminal}

Parallel to Section~\ref{sec:outer}, 
in this section we shall present three characterisations of terminal planar galled networks, each given by a different type of forbidden structure: a forbidden vertex configuration, forbidden undirected subgraphs, and forbidden directed subgraphs.

We first prove an elementary lemma needed for the first characterisation.
To this end, we consider a pivotal basis of a subgraph $H$ of $\overline{N}$ for a galled network $N$, as introduced in Section~\ref{sec:background}. Then clearly every reticulation vertex $r$ in $H$ is cut-visible from $H$ in $\overline{N}$ as $r$ is itself a cut vertex of $\overline{N}$ by Lemma~\ref{lmm:retcutvertex}. This implies that all the reticulation vertices of $H$ form a pivotal basis of $H$ in $\overline{N}$.
We shall need the following properties of pivotal bases.

\begin{lemma}\label{lmm:cut-visible-basic}
Let $H$ be a subgraph of a graph $G$, and let $B$ be a pivotal basis of $H$ in $G$. Then the following assertions hold:
\begin{enumerate}
\item[(i)] If $B'$ is a subset of $B$, then  $B'$ is also a pivotal basis of $H$ in $G$.
\item[(ii)] If $H'$ is a subgraph of $H$, then $V(H')\cap B$ is a pivotal basis of  $H'$ in $G$.
\item[(iii)] If a vertex $v$ in $H$ is a cut vertex of $G$,   then $B \cup \{v\}$ is also a pivotal basis of $H$ in $G$. 
\item[(iv)] If $G=\overline{N}$ for a galled network $N$, and $r$ is a reticulation vertex of $H$, then $B\cup\{r\}$ is also a pivotal basis of $H$ in $\overline{N}$.
\end{enumerate}
\end{lemma}

\begin{proof}
Let $B = \{v_1, \dots, v_k\} \subseteq V(H)$ and $P_1, \dots, P_k$ be pivotal rays in $G$ such that, for each $1 \le i \le k$, $P_i$ connects $v_i \in B$ to a cut vertex of $G$, $V(P_i) \cap V(H) = \{v_i\}$, and distinct rays $P_i$ and $P_j$ are vertex-disjoint.

For (i): reindexing if necessary, assume that $B' = \{v_1, \dots, v_j\}$ for $j\leq k$. Since $H$ remains unchanged, we still have $V(P_i) \cap V(H) = \{v_i\}$ for all $1 \le i \le j$ and $P_1, \dots, P_j$ are still pairwise vertex-disjoint. Hence $B'$ is also a pivotal basis of $H$ in $G$.

For (ii): since $V(H') \subseteq V(H)$, we have $V(P_i) \cap V(H') \subseteq V(P_i) \cap V(H) = \{v_i\}$, and $v_i \in V(H') \cap B$ gives $V(P_i) \cap V(H')= \{v_i\}$. Hence the paths $P_i$ with $v_i\in V(H')\cap B$ show that $V(H')\cap B$ is a pivotal basis of $H'$ in $G$.

For (iii): set $v_{k+1} := v$ and let $P_{k+1}$ be the degenerate path consisting of $v$. If $v\in B$, there is nothing to prove. Thus assume $v\notin B$ and $v \neq v_i$ for $1 \leq i \leq k$. Then we have $v \notin V(P_i)$ since $V(P_i) \cap V(H) = \{v_i\}$. Thus $V(P_{k+1}) \cap V(P_i) = \emptyset$ and $B \cup \{v\}$ is a pivotal basis of $H$ in $G$.

For (iv), it follows directly from (iii) and Lemma~\ref{lmm:retcutvertex}.
\end{proof}

Let $[T,B]$ denote the {\em pivotal pair} consisting of a graph $T$ together with a specified vertex subset $B$ of $V(T)$. Then a \emph{pivotal subdivision} of the pair  $[T,B]$ in a graph $G$ is a subgraph $H$ of $G$ such that $H$ is a subdivision of $T$, and the branch vertices of $H$ corresponding to the vertices in $B$ form a pivotal basis of $H$ in $G$. Note that the usual notion of a subdivision of a graph $T$ in the literature is a pivotal subdivision of the pair $[T,\emptyset]$. 

Two particularly important pivotal pairs are used in this paper. First, we denote by $K_{2:3}$ the pivotal pair $[K_{2,3},B]$ depicted in the left panel of Figure~\ref{terminal forbid}, where $K_{2,3}$ is the complete bipartite graph with bipartition $(B,B')$, where $B$ consisting of the three degree-$2$ vertices, and its complement $B'$ consisting of two vertices of degree-$3$. Furthermore, we denote by $K_{4:2}$ the pivotal pair  $[T,B]$ in the right panel of Figure~\ref{terminal forbid}, where  $T$ is obtained from $K_4$ by subdividing each of two non-adjacent edges (i.e., two edges with no common endpoint) once, the set $B$ consists of the two resulting degree-$2$ vertices, while the complement of $B$ consists of the four vertices of degree-$3$. 
 With these definitions, we state the forbidden subgraph characterisation of terminal planarity in terms of the underlying graph.

\begin{figure}[ht]
\centering
\includegraphics[scale=0.55]{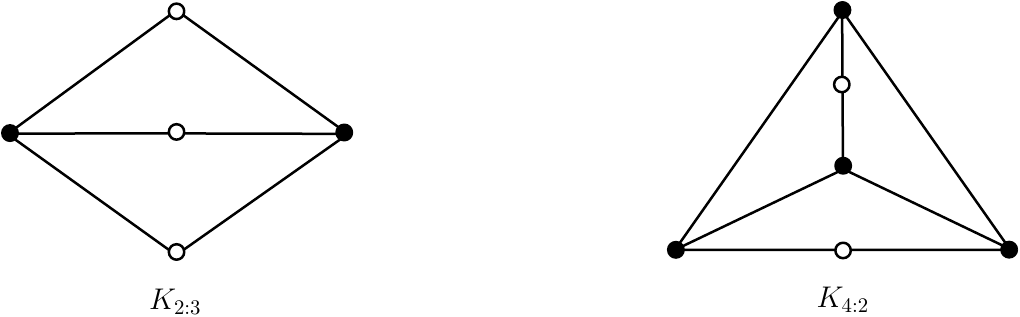}
\caption{Two examples of pivotal pairs $(T,B)$:  $K_{2:3}$ (left) and ${K}_{4:2}$ (right), in which the subset $B$ consists of the white vertices. }
\label{terminal forbid}
\end{figure}

\begin{theorem}\label{thm:terminal_forbid-te}
A phylogenetic network $N$ is terminal planar if and only if  $\overline{N}$ contains no subdivision of $K_{3,3}$, no pivotal subdivision of $K_{2:3}$, and no pivotal subdivision of ${K}_{4:2}$.
\end{theorem}

The proof of the above theorem is  included in the Appendix, which is based on a slightly stronger reformulation of \cite[Corollary~5.16]{Miyaji2026}.

For galled networks, the three forbidden structures in Theorem~\ref{thm:terminal_forbid-te} are not independent. The next two lemmas show that both a subdivision of $K_{3,3}$ and a pivotal subdivision of $K_{4:2}$ force a pivotal subdivision of $K_{2:3}$.

\begin{figure}[htpb]
\centering
\includegraphics[width=0.9\textwidth]{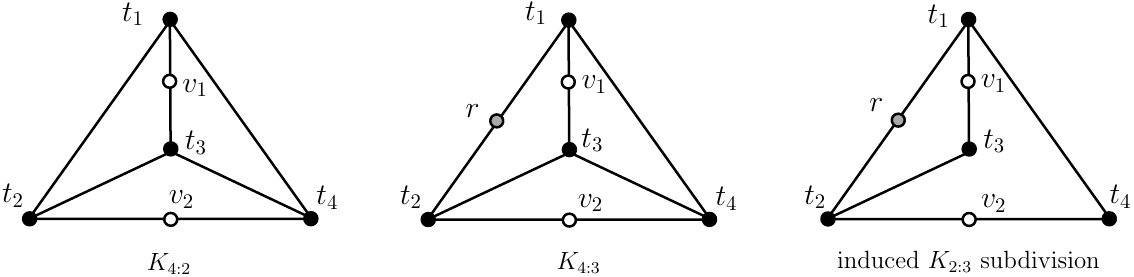}
\caption{Reduction from a pivotal subdivision $H$ of $K_{4:2}$ (left) to a pivotal subdivision $H'$ of $K_{2:3}$ (right), through an intermediate pivotal subdivision of $K_{4:3}$ (middle). $\{v_1,v_2\}$ is a pivotal basis of $H$, and $\{r,v_1,v_2\}$ is a pivotal basis of $H'$;}
\label{fig:s2_to_s1}
\end{figure}

\begin{lemma}\label{lem:k42_to_k23}
Let $N$ be a galled network. If $\overline{N}$ contains a pivotal subdivision of $K_{4:2}$, then $\overline{N}$ contains a pivotal subdivision of $K_{2:3}$.
\end{lemma}

\begin{proof}
Suppose $\overline{N}$ contains a pivotal subdivision $H$ of $K_{4:2}$, with four degree-$3$ branch vertices $t_1,t_2,t_3,t_4$ (which are necessarily tree vertices by Lemma~\ref{lmm:deg3-tree}) and two degree-$2$ branch vertices $v_1,v_2$ such that $\{v_1,v_2\}$ is a pivotal basis of $H$ in $\overline{N}$. 
Let $P_{i,j}$ denote the canonical path corresponding to the edge $\{t_i,t_j\}$ of $K_{4}$, where $v_i$ is an internal vertex of $P_{i,i+2}$ for $1\le i\le 2$.

Consider the cycle $C=P_{1,2}\cup P_{2,3}\cup P_{3,4}\cup P_{1,4}$ in $H$. As the cycle $C$ must contain a reticulation vertex $r$ by Lemma~\ref{lmm:cycle}, without loss of generality, we may assume that $r$ lies on $P_{1,2}$. Thus, we may view $H$ as a subdivision of the graph obtained from $K_4$ by subdividing once each of the three edges $\{t_1,t_3\}$, $\{t_1,t_2\}$, and $\{t_2,t_4\}$, where the corresponding degree-$2$ branch vertices are $v_1,r,v_2$, respectively. By
Lemma~\ref{lmm:cut-visible-basic}(iv), $\{r,v_1,v_2\}$ is a pivotal basis of $H$ in $\overline{N}$.

Removing the edges of $P_{3,4}$ in $H$ leaves three internally disjoint paths between $t_1$ and $t_2$: $Q_1=P_{1,2}$ with $r$, $Q_2=P_{1,3}\cup P_{2,3}$ with $v_1$, and $Q_3=P_{1,4}\cup P_{2,4}$ with $v_2$. Their union $H'=Q_1 \cup Q_2 \cup Q_3$ is a subdivision of $K_{2:3}$ in $\overline{N}$, with the two degree-$3$ branch vertices in set $A$ corresponding to $t_1,t_2$ and the three degree-$2$ branch vertices in set $B$ corresponding to $r,v_1,v_2$.
Since $\{r,v_1,v_2\}\subseteq V(H')$ and $H'$ is a subgraph of $H$, Lemma~\ref{lmm:cut-visible-basic}(ii) implies that $\{r,v_1,v_2\}$ is a pivotal basis of $H'$ in $\overline{N}$. Hence $H'$ is a pivotal subdivision of $K_{2:3}$ in $\overline{N}$.
\end{proof}

\begin{lemma}\label{lem:k33_to_k23}
Let $N$ be a galled network. If $\overline{N}$ contains a subdivision of $K_{3,3}$, then $\overline{N}$ contains a pivotal subdivision of $K_{2:3}$.
\end{lemma}

\begin{proof}
Suppose $\overline{N}$ contains a subdivision $H$ of $K_{3,3}$ with branch vertices $p_1,p_2,p_3$ and $q_1,q_2,q_3$ corresponding to the two bipartition classes, and let $P_{i,j}$ be the canonical path corresponding to the edge $\{p_i,q_j\}$ in $K_{3,3}$. Lemma~\ref{lmm:deg3-tree} shows that every branch vertex in $H$ is a tree vertex; consequently any reticulation vertex lying on a canonical path $P_{i,j}$ is an internal vertex.

Call a canonical path $P_{i,j}$ marked if it contains a reticulation vertex. We claim that at least four canonical paths in $H$ are marked. Suppose not, so at most three canonical paths are marked. Removing the corresponding edges from $K_{3,3}$, while keeping all six vertices, leaves a graph with six vertices and at least six edges. Hence this graph contains a cycle, since any acyclic graph on six vertices has at most five edges. Replacing the edges of this cycle by the corresponding canonical paths in $H$ yields a cycle containing no reticulation vertex, contradicting Lemma~\ref{lmm:cycle}.

Say that a pair $\{a,b\}$ of branch vertices on one side covers the other side if, for every branch vertex $v$ on the other side, some marked canonical path has $v$ as one endpoint and a vertex of $\{a,b\}$ as the other.

\smallskip
\noindent\emph{Case 1: some pair covers the other side.}

Say $\{p_1,p_2\}$ covers the $q$-side; by definition, for each $j\in\{1,2,3\}$ at least one of $P_{1,j},P_{2,j}$ is marked. Then $Q_j:=P_{1,j}\cup P_{2,j}$ ($j=1,2,3$) are three internally disjoint paths between $p_1$ and $p_2$, each containing a reticulation vertex.

\smallskip
\noindent\emph{Case 2: no pair covers the other side.}

In this case, we show that the marked paths have endpoints among only two of $p_1,p_2,p_3$ and only two of $q_1,q_2,q_3$. Suppose instead that $p_1,p_2,p_3$ are all endpoints of marked canonical paths. Each marked canonical path has a unique endpoint among $q_1,q_2,q_3$; as there are at least four such paths, some $q_j$, say $q_1$, is an endpoint of at least two marked paths, two of which have distinct other endpoints, say $p_1$ and $p_2$, so $P_{1,1}$ and $P_{2,1}$ are marked. Let $P_{3,i}$ be a marked path with endpoint $p_3$. If $i\neq 1$ then the pair $\{q_1,q_i\}$ covers the $p$-side; if $i=1$ then $P_{1,1},P_{2,1},P_{3,1}$ are marked, hence $\{q_1,q_j\}$ covers the $p$-side for any $j \neq 1$. Either contradicts the hypothesis, proving the claim. Symmetrically, only two of $q_1,q_2,q_3$ are endpoints of marked paths. Hence the marked paths lie among the four canonical paths between two $p$'s and two $q$'s; relabelling if necessary, they are $P_{1,1},P_{1,2},P_{2,1},P_{2,2}$. Set $Q_1=P_{2,2}$, $Q_2=P_{2,1}\cup P_{3,1}\cup P_{3,2}$, and
$Q_3=P_{2,3}\cup P_{1,3}\cup P_{1,2}$. Then $Q_1,Q_2,Q_3$ are three internally disjoint paths between $p_2$ and $q_2$, containing the marked paths $P_{2,2},P_{2,1},P_{1,2}$, respectively.

\smallskip
In either case $H'=Q_1\cup Q_2\cup Q_3$ is a subdivision of $K_{2,3}$ in which each of $Q_i$ contains a reticulation vertex. For each $i$, choose a reticulation vertex $r_i$ that is an internal vertex on $Q_i$, so that $r_1,r_2,r_3$ are distinct degree-$2$ branch vertices of $H'$. By Lemma~\ref{lmm:retcutvertex} each $r_i$ is a cut vertex, hence
$\{r_1,r_2,r_3\}$ is a pivotal basis of $H'$ in $\overline{N}$, and $\overline{N}$ contains a pivotal subdivision of $K_{2:3}$.
\end{proof}

The three forbidden structures in Theorem~\ref{thm:terminal_forbid-te} can be reduced to a single one for galled networks.

\begin{lemma}\label{lem:no_s2}
A galled network $N$ is terminal planar if and only if $\overline{N}$ contains no pivotal subdivision of ${K}_{2:3}$.

\end{lemma}
\begin{proof}
By Theorem~\ref{thm:terminal_forbid-te} and Lemmas~\ref{lem:k42_to_k23} and~\ref{lem:k33_to_k23}, the absence of a  pivotal subdivision of $K_{2:3}$ in $\overline{N}$ implies the absence of a pivotal subdivision of $K_{4:2}$ and a subdivision of $K_{3,3}$. Hence the conclusion follows.
\end{proof}

Recalling from Section~\ref{sec:outer} the definition of a normal subdivision of $K_{2,3}$, we say $H$ is a \emph{normal pivotal subdivision} of $K_{2:3}$ in a graph $G$ if $H$ is both a normal subdivision of $K_{2,3}$ and a pivotal subdivision of $K_{2:3}$. As in Section~\ref{sec:outer}, we provide a stronger version of Lemma~\ref{lem:no_s2}.

\begin{theorem}\label{thm:normal-Kk23}
Suppose $N$ is a galled network. Then $N$ is terminal planar if and only if $\overline{N}$ contains no normal pivotal subdivision of $K_{2:3}$.
\end{theorem}

\begin{proof}
By Lemma~\ref{lem:no_s2}, it is sufficient to prove that if $\overline{N}$ contains a pivotal subdivision $H$ of $K_{2:3}$ as a subgraph, then  $\overline{N}$ contains a normal pivotal subdivision of $K_{2:3}$.
We use the notation from the proof of Theorem~\ref{thm:normalk23}. In particular, $t_1$ and $t_2$ are tree vertices, and we may assume that $H$ is not normal and $w_1=v_1$.

If $P$ is internally disjoint from $P_2$ and $P_3$, then $H'=P \cup P_2 \cup P_3$ is a normal subdivision of $K_{2,3}$. 
Choose indices $p<q$ such that $w_p,w_q\in V(P_1)$, no vertex $w_k$ with $p<k<q$ lies on $P_1$, and $P_1[w_p,w_q]$ contains a reticulation vertex which is a cut vertex in $\overline{N}$; otherwise $P_1[w_p,w_q] \cup P[w_p,w_q] $ would form a cycle consisting entirely of tree vertices, contradicting Lemma~\ref{lmm:cycle}.
Let $v_c$ be the closest cut vertex to $w_p$ on $P_1[w_p,w_q]$. Then $P_1[w_p, v_c] \cap V(H')=\{w_p\}$ and $v_c$ is the only cut vertex of $P_1[w_p, v_c]$, so $w_p$ is cut-visible from $H'$ in $\overline{N}$. By Lemma~\ref{lmm:cut-visible-basic}(iv), we obtain that $\{w_p, r_2, r_3\}$ is a pivotal basis of $H'$, where $r_2$ and $r_3$ are the reticulation vertices on $P_2$ and $P_3$, respectively. Hence $H'$ is a normal pivotal subdivision of $K_{2:3}$.

It remains to consider the case that $P$ has an internal vertex on $P_2\cup P_3$. Let $u=w_j$, $Q_1$, $Q_2$, and $Q_3$ be defined as in the proof of Theorem~\ref{thm:normalk23}. Let $i<j$ be maximal such that $w_i\in V(P_1)$. Let $e$ be the edge of $P_1$ incident with $w_i$ and not contained in $P[t_1,w_i]$. Such $e$ exists, otherwise the path $P$ is self-intersecting. Let $P_1[w_i,w_h]$ be the maximal subpath of $P_1$ containing $e$ whose internal vertices do not lie on $P$. Note that $w_i$, $w_h$, and $u$ are all tree vertices by Lemma~\ref{lmm:deg3-tree}, and each of $P_1[w_i,w_h]$ and $Q_3$ contains a reticulation vertex as an internal vertex, say $r'_1$ and $r'_3$, respectively. Otherwise, either $P_1[w_i,w_h]\cup P[w_i,w_h]$ or $Q_3\cup P[t_1,u]$ would form a cycle consisting entirely of tree vertices. Arguing as in the previous case, we see that $\{w_i, r_2, r'_3\}$ is a pivotal basis of $H'$, where $r_2$ is a reticulation vertex on $P_2$. Hence $H'$ is a normal pivotal subdivision of $K_{2:3}$.
\end{proof}

Next, we turn to the characterisation in terms of a forbidden vertex configuration. We shall use the following form of the Fan Lemma.

\begin{lemma}[Fan Lemma, Dirac 1960 \cite{Dirac1960}]\label{lmm:fan}
Let $G$ be a $k$-connected graph, and let
$v_0,v_1, \ldots, v_k$ be $k+1$ distinct vertices in $V(G)$. Then there exist $k$ paths $P_i$ between $v_0$ and $v_i$ for $1\le i \le k$ such that $V(P_i)\cap V(P_j)=\{v_0\}$ holds for $1\le i <j\le k$.
\end{lemma}

\begin{lemma}\label{lmm:cutvis-block}
Let $H$ be a biconnected subgraph of a graph $G$ with $|V(H)|\ge 2$. 
Suppose $R=(v_0,v_1,\dots,v_k)$ is a pivotal ray of $H$ in $G$ with $k\ge 1$, and denote the block of $G$ containing $H$ by $B$.
Then $v_0$ is not a cut vertex of $G$, and the ray $R$ is contained in $B$ (i.e., $V(R)\subsetneq V(B)$). 
Furthermore, there exist a vertex $u\in V(H)\setminus\{v_0\}$ and a path $P$ between $v_0$ and $u$ such that $v_k\in V(P)$ and $V(P)\cap V(H)=\{v_0,u\}$.

\end{lemma}

\begin{proof}
Since $R$ is non-degenerate, it follows that $v_i$ is not a cut vertex for $0\le i \le k-1$. Hence all edges incident to $v_i$ belong to the same block (see Diestel~\cite{Diestel}, Section~3.1) for $0\le i \le k-1$, which implies that $V(R)\subset V(B)$.

Consider two distinct vertices $u_1,u_2$ of $H$. Applying the Fan Lemma (Lemma~\ref{lmm:fan}) to the vertices $v_k,u_1,u_2$ in the block $B$ shows that there exist two internally disjoint paths $P_1$ and $P_2$ from $v_k$ to $u_1$ and $u_2$, respectively.
For $i=1,2$, let $w_i$ be the first vertex of $P_i$ belonging to $V(H)$ (starting from $v_k$) and put $Q_i=P_i[v_k,w_i]$. Then $w_1\neq w_2$, $V(Q_i)\cap V(H)=\{w_i\}$, and $V(Q_1)\cap V(Q_2)=\{v_k\}$. Now we consider the following two subcases:

In the first case, $v_0\in \{w_1,w_2\}$. Then relabelling if necessary, we can assume $w_1=v_0$.
Then $P=Q_1\cup Q_2$ is a path between $v_0$ and $w_2$ with $v_k\in V(P)$ and $V(P)\cap V(H)=\{v_0,w_2\}$, as required.

In the second case, $v_0\not \in \{w_1,w_2\}$. 
Now consider the set 
$V'=V(R)\cap (V(Q_1)\cup V(Q_2))$, which is nonempty as it contains $v_k$. 
Let $j$ be the smallest index in $\{1,\dots,k\}$ such that $v_j\in V'$.
Swapping the indices if necessary, we may assume $v_j\in V(Q_1)$.
Now consider the path $P=R[v_0,v_j]\cup Q_1[v_j,v_k]\cup Q_2[v_k,w_2]$ (using the convention that $Q_1[v_k,v_k]$ contains no edge).
Then 
$P$ is the desired path between $v_0$ and $w_2$ with $v_k\in V(P)$. Indeed, we have $V(P)\cap V(H)=\{v_0,w_2\}$ in view of $V(R)\cap V(H)=\{v_0\}$ and $V(Q_i)\cap V(H)=\{w_i\}$.
\end{proof}

We now prove the forbidden vertex configuration characterisation of terminal planar galled networks.

\begin{theorem}\label{noter2gall}
A galled network $N$ is terminal planar if and only if it contains no primitive tree vertex.
\end{theorem}

\begin{proof}
($\Rightarrow$) Suppose that $v$ is a primitive tree vertex of ${N}$, and let $G_1, G_2$ be two galled cycles such that $v$ is an internal vertex on their intersecting path $P$ with endpoints $t_1, t_2$. Let $r_1, r_2$ be the reticulations of $G_1, G_2$. Then $G_1 \cup G_2$ is a subdivision of $K_{2,3}$ in $\overline{N}$ with degree~$3$ branch vertices $t_1, t_2$, and degree~$2$ branch vertices $v, r_1, r_2$. By Lemma~\ref{lmm:overlap-path-property}(ii) and Lemma~\ref{lmm:cut-visible-basic}(iv), $\{v,r_1,r_2\}$ is a pivotal basis of $G_1\cup G_2$ in $\overline{N}$. Hence $G_1\cup G_2$ is a pivotal subdivision of $K_{2:3}$ in $\overline{N}$, and $N$ is not terminal planar by Lemma~\ref{lem:no_s2}.
\smallskip

($\Leftarrow$) Suppose $N$ is not terminal planar. By Theorem~\ref{thm:normal-Kk23}, $\overline{N}$ contains a normal pivotal subdivision $H$ of $K_{2:3}$ with branch vertices $t_1, t_2$ and three internally disjoint paths $P_1, P_2, P_3$. Without loss of generality, assume that $P_3$ is the unique tree path within $H$. Let $v$ be the branch vertex of $H$ corresponding to the vertex of $B$ that lies on $P_3$, and let $e=\{v,v'\}$ be the edge of $\overline{N}$ incident to $v$ such that $e \notin E(H)$. We show that ${N}$ contains a primitive tree vertex.

\smallskip
\noindent\textit{Case 1: $v$ is a cut vertex of $\overline{N}$.}

Applying Theorem~\ref{cycleandgalledcycle} to the cycles $P_1\cup P_3$ and $P_2\cup P_3$, we obtain reticulation vertices $r_1\in V(P_1)$ and $r_2\in V(P_2)$ whose galled cycles $G_1$ and $G_2$ both contain $v$. Since the cut edge $e$ is not in any cycle, $G_1$ and $G_2$ both contain the two incident edges of $v$ on $P_3$, so $v$ is an internal vertex on the intersecting path of $G_1$ and $G_2$, and hence $v$ is a primitive tree vertex.

\smallskip
\noindent\textit{Case 2: $v$ is not a cut vertex of $\overline{N}$.}

Note that $\overline{N}-v$ is connected. Let $\mathcal{P}$ be the set of paths in $\overline{N}$ between $v$ and some vertex of $V(H)\setminus\{v\}$ whose internal vertices do not lie in $H$. By Lemma~\ref{lmm:cutvis-block}, $\mathcal{P}$ is non-empty. Moreover, every path in $\mathcal{P}$ contains $e$, and both endpoints of every path in $\mathcal{P}$ are tree vertices by Lemma~\ref{lmm:deg3-tree}. We consider the following three cases.

\smallskip
\noindent\emph{Case 2(i): every path in $\mathcal{P}$ contains a reticulation vertex.}

By Theorem~\ref{cycleandgalledcycle} as in Case~1, there exist galled cycles $G_1$ and $G_2$, both containing $v$, with reticulation vertices on $P_1$ and $P_2$, respectively. If $e \in E(G_1)$, then $G_1$ contains a subpath $P$ in $\mathcal{P}$, which by assumption contains an internal reticulation vertex on $P$ and is distinct from $r_1$, contradicting that $G_1$ is a galled cycle. Similarly, $e\notin E(G_2)$. Hence both $G_1$ and $G_2$ contain the two incident edges of $v$ on $P_3$, and $v$ is a primitive tree vertex.

\smallskip
\noindent\emph{Case 2(ii): every path in $\mathcal{P}$ is a tree path.}

Let $R=(v_0,v_1,\dots,v_k)$ be the non-degenerate pivotal ray of $H$ in $\overline{N}$ with $v_0=v$ and $v_k=w$. By Lemma~\ref{lmm:cutvis-block}, there exists a path $P\in\mathcal{P}$ such that $w\in V(P)$. Note that $P$ is a tree path by assumption, hence $w$ is a tree vertex.  Let $u$ be the endpoint of $P$ in $V(H)\setminus\{v\}$.  If $u\in V(P_3)$, then $P \cup P_3[v,u]$ is a cycle consisting entirely of tree vertices, contradicting Lemma~\ref{lmm:cycle}. Without loss of generality, we may assume that $u\in V(P_1)$.

Consider the two cycles
$C_1=P\cup P_3[t_1,v]\cup P_1[t_1,u]$
and
$C_2=P\cup P_3[t_1,v]\cup P_2\cup P_1[u,t_2]$.
Both cycles contain $w$. Since $P$ and $P_3$ are tree paths, Lemma~\ref{lmm:cycle} and Theorem~\ref{cycleandgalledcycle} yield two galled cycles $G_1$ and $G_2$ containing $w$, whose associated reticulation vertices lie on $P_1[t_1,u]$ of $C_1$ and $P_1[u,t_2]\cup P_2$ of $C_2$, respectively.

Since $w$ is a cut vertex while $P \cup H$ is 2-connected, the edge incident with $w$ that does not belong to $P$ is a cut edge of $\overline{N}$, and hence cannot lie in either $G_1$ or $G_2$. Therefore $G_1$ and $G_2$ both contain the two incident edges of $w$ on $P$. Hence $w$ is an internal vertex on the intersecting path of $G_1$ and $G_2$, and so $w$ is a primitive tree vertex.

\smallskip
\noindent\emph{Case 2(iii): $\mathcal{P}$ contains both kinds of paths.}

Take $P^* \in \mathcal{P}$ containing a reticulation vertex. Let $r^*$ be the closest reticulation vertex to $v$ on $P^*$, and write
$P^*[v,r^*]=(w_0,\dots,w_k)$ with $w_0=v$ and $w_k=r^*$. Let $i$ be the largest index with $i<k$ such that $w_i$ lies on some tree path $P\in\mathcal{P}$. Note that $i\ne 0$; otherwise, $v'=w_1=r^*$ would be a reticulation vertex contained in every path in $\mathcal{P}$. Let $u$ be the endpoint other than $v$  of the tree path $P$. If $u\in V(P_3)$, then $P \cup P_3[u,v]$ is a cycle consisting entirely of tree vertices, contradicting Lemma~\ref{lmm:cycle}. Hence, without loss of generality, we may assume that $u\in V(P_1)$.

Consider the two cycles
$C_1=P_3[t_1,v]\cup P\cup P_1[t_1,u]$
and
$C_2=P_3[t_1,v]\cup P\cup P_1[u,t_2]\cup P_2$.
As in Case~2(ii), these cycles yield two galled cycles $G_1$ and $G_2$ containing $w_i$, whose associated reticulation vertices lie on $P_1[t_1,u]$ of $C_1$ and $P_1[u,t_2]\cup P_2$ of $C_2$, respectively. Note that $P_3[t_1,v]\cup P$ is a tree path with the internal tree vertex $w_i$. Set $H'=C_1\cup C_2$ and $e'=\{w_i,w_{i+1}\}$. By the maximality of $i$, every path from $w_i$ to $H'$ containing $e'$ contains a reticulation vertex. Hence the argument of Case~2(i) applies to $H'$ and $e'$, and $G_1$ and $G_2$ both contain the two incident edges of $w_i$ on $P$. Hence $w_i$ is a primitive tree vertex.
\end{proof}

Note that Theorem~\ref{noter2gall} implies that a planar galled network $N$ is terminal planar if and only if any two distinct galled cycles in $N$ are either disjoint or share exactly a single tree edge. Furthermore, this theorem enables us to obtain an additional characterisation  via forbidden directed subgraphs; the proof is similar to that of Theorem~\ref{outer forbidden} for outerplanar galled networks, and hence some details are omitted.

\begin{theorem}\label{terminal forbidden}
A galled network $N$ is terminal planar if and only if $N$ contains no directed subdivision of $D_1$ or $D_2$ shown in Figure~\ref{fig:mainforbidden} as a subgraph.
\end{theorem}

\begin{proof}
By Theorem~\ref{noter2gall}, it suffices to show the claim that $N$ contains a primitive tree vertex if and only if $N$ contains a directed subdivision of $D_1$ or $D_2$. Indeed, this claim is already implicitly established in the proof of Theorem~\ref{outer forbidden}: the forward direction is straightforward as both $D_1$ and $D_2$ (see Figure~\ref{fig:mainforbidden}) contain primitive tree vertices, and the other direction follows from Case~(i) in the proof of Theorem~\ref{outer forbidden}.
\end{proof}

\section{Conclusion and Open Problems} \label{sec:conclusion}

In this paper, we presented three characterisations of both outerplanar and terminal planar galled networks in terms of forbidden vertex configurations, forbidden directed subgraphs, and forbidden structures in their associated underlying undirected graphs. It would be of interest to explore the algorithmic implications of these characterisations. Furthermore, it would be of interest to investigate whether the first two characterisations can be extended to upward galled networks, which may offer further insights into  a complete characterisation of upward phylogenetic networks, an open problem posed in~\cite{Miyaji2026}. Finally, future work may also consider the associated counting problems and the characterisation of other important subclasses of phylogenetic networks~\cite{liu2025asymptotic}.

\bmhead{Acknowledgements}
TW would like to thank Prof. Andreas Dress for introducing him to combinatorial phylogenetics twenty years ago, an influence that has profoundly shaped his research journey. HL is supported by JSPS KAKENHI Grant Number JP25KJ2168, and he thanks his PhD advisor Dr. Momoko Hayamizu for support and guidance. GRY is supported by NSTC Grant Number NSTC-113-2115-M-110-004-MY3. We all thank Michael Wallner for constructive discussions and comments. 

\begin{appendices}

\section{Cut-labelled graph characterisation}\label{secA1}

We briefly recall the terminology of cut-labeled graphs used in
\cite{Miyaji2026}. A \emph{$0/1$-labeled graph} is a pair $(G,g)$,
where $G$ is a graph and $g:V(G)\cup E(G)\to \{0,1\}$ is a labeling function assigning each vertex and edge a label of 0 or 1.
Given a degree-$2$ vertex $v$ in $G$ with two incident edges $e_i=\{v,v_i\}$ for $i=1,2$ such that $g(e_1)=g(e_2)$, the \emph{label-preserving smoothing} of $v$ produces a $0/1$-labeled graph $(G',g')$ such that $V(G'):=V(G)\setminus \{v\}$ and $E(G'):=(E(G)\cup \{e'\})\setminus \{e_1,e_2\}$, where $e'=\{v_1,v_2\}$. The labeling function $g'$ is defined by $g'(e'):=g(e_1)=g(e_2)$ and $g'(x):=g(x)$ for every other $x\in V(G')\cup E(G')$.

Given another $0/1$-labeled graph $(F,f)$, we say $(F,f)$ is a \emph{label-preserving subgraph} of $(G,g)$ if $F$ is a subgraph of $G$ and $f(x)=g(x)$ for each $x\in V(F)\cup E(F)$. On the other hand, a graph isomorphism is \emph{label-preserving} if it preserves
both the graph structure and all vertex and edge labels.

For a graph $G$, its \emph{cut-labeled graph} is the
$0/1$-labeled graph $L(G)=(G,\ell)$, where $\ell(x)=1$ for $x \in V(G)\cup E(G)$ if and only if $x$ is a cut vertex or cut edge of $G$. 
A characterisation of terminal planar phylogenetic networks using forbidden cut-labeled subgraphs was given in \cite{Miyaji2026}. 
Let $\mathcal H$ denote the family of eight $0/1$-labeled graphs shown in Figure~\ref{fig:binary_H_structures}. 
Slightly rephrased, for $i=1,2,3$,
a $0/1$-labeled graph $(G,g)$ is said to \emph{contain an $\mathcal H$ structure} if there exists a $0/1$-labeled graph $(G^*,g^*)$ obtained from a label-preserving subgraph $(G',g')$ of $(G,g)$ by a finite sequence of label-preserving vertex smoothings, such that $(G^*,g^*)$ is label-preserving isomorphic to some graph in $\mathcal H$. More specifically, for a fixed $(H,h)\in\mathcal H$, we say that $(G,g)$ \emph{contains an $(H,h)$ structure} if $(G^*,g^*)$ above can be label-preserving isomorphic to $(H,h)$.

For later use, we introduce a refined family $\mathcal H^*$ of such $\mathcal H$ structures, in which the labeling function is further restricted.

That is, we set 
\[\mathcal H^*=
\{(H_{1},h^*),\quad (H_{2,0},h^*),\ldots,(H_{2,3},h^*),\quad
(H_{3,0},h^*),(H_{3,1},h^*),(H_{3,2},h^*)\}.\] This family contains precisely eight $0/1$-labeled graphs.

\begin{figure}[htbp]
    \centering
    \includegraphics[width=0.9\textwidth]{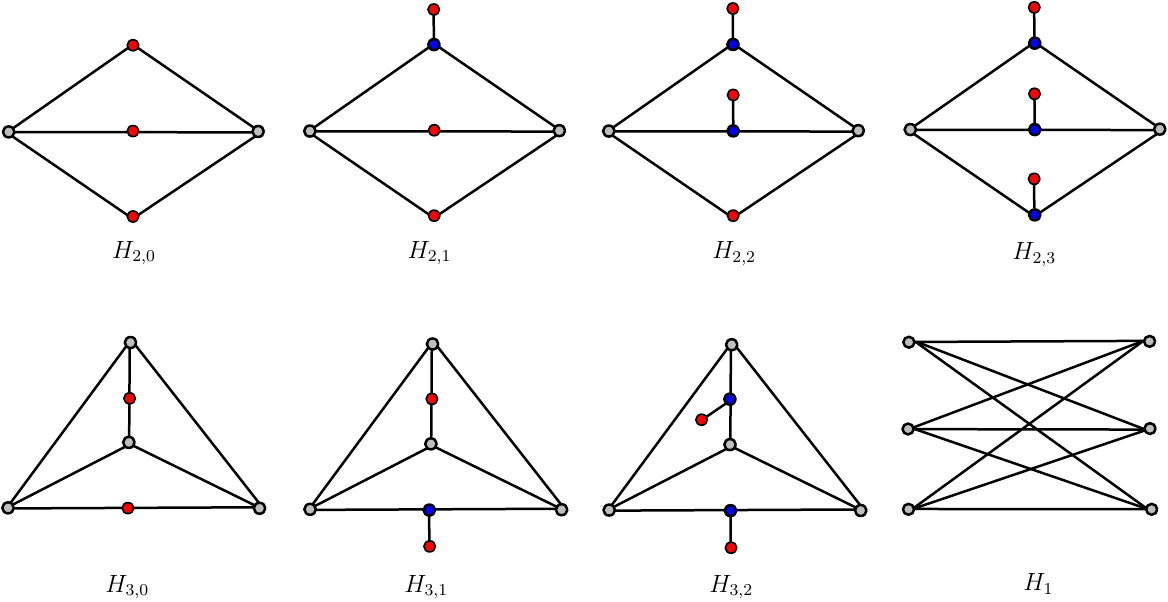} 
    \caption{Eight forbidden structures for binary phylogenetic networks. 
    The set of $\mathcal H$ contains $0/1$-labeled graphs $(H,h)$ such that $H$ is one of the eight graphs depicted here, and $h$ maps all edges and blue vertices to $0$, red vertices to $1$, while there are no constraints on gray vertices (i.e., they can have label $0$ or $1$). 
    Note that $\mathcal H^*$ is a subset of $\mathcal H$ consisting of graphs $(H,h^*)$ in which $h^*$ maps gray vertices to $0$. }
    \label{fig:binary_H_structures}
\end{figure}

\begin{lemma}\label{lem:H2_degree3_label0}
Let $N$ be a binary phylogenetic network. 
Suppose that $L(\overline{N})$ contains an $(H,h)$ structure for some $(H,h)\in\mathcal H$.
Then $h(v)=0$ for every degree-$3$ vertex $v$ of $H$ that is not incident to a cut edge in $H$.
Consequently, $(H,h)$ is contained in $\mathcal H^*$.
\end{lemma}

\begin{proof}
By definition, there exist a label-preserving subgraph $(F,\ell|_F)$ of $L(\overline{N})$ and a $0/1$-labeled graph $(F_1,\ell_1)$ obtained from $(F,\ell|_F)$ by a finite sequence of label-preserving vertex smoothings such that $(F_1,\ell_1)$ is label-preserving isomorphic to $(H,h)$. We identify $(H,h)$ with $(F_1,\ell_1)$ via this isomorphism.

Let $V_0(H)$ be the set of degree-$3$ vertices  of $H$ that are not incident to a cut edge in $H$. Then we have $|V_0(H)|\ge 2$ by checking each graph  in $\mathcal H$. By assumption, we have $v\in V_0(H)$. Now consider another vertex $u \in V_0(H)$. Then there exist three internally disjoint paths in $H$ between $v$ and $u$, which can be checked by inspecting each graph  in $\mathcal H$.  

Let $v'$ and $u'$ be the vertices in $F_1$ corresponding to $v$ and $u$, respectively.  Then $v'$ has degree $3$ in $F_1$, and it suffices to show that $\ell_1(v')=0$. Since $V(F_1)\subseteq V(F)\subseteq V(\overline{N})$, it follows that  $v'$ and $u'$ both are also vertices of $\overline{N}$. Since $v'$ has degree $3$ in $F_1$ and $N$ is binary, we have
$\deg_{\overline{N}}(v')=3.$ It remains to show that $v'$ is not a cut vertex of $\overline{N}$. Indeed, if $v'$ were a cut vertex, then, since $\deg_{\overline{N}}(v')=3$, there are at most two internally disjoint paths between $v'$ and $u'$ in $\overline{N}$. Since $F$ is a subgraph of $\overline{N}$, there are also at most two internally disjoint paths between $v'$ and $u'$ in $F$. This contradicts the three internally disjoint paths between $v'$ and $u'$ in $F$ corresponding to the three internally disjoint paths between $v$ and $u$ in $H$. Hence $\ell_1(v')=h(v)=0$.
\end{proof}

Now we present a slightly stronger statement of \cite[Corollary~5.16]{Miyaji2026}. 

\begin{theorem}
\label{thm:cut-label:str}
Let $N$ be a binary phylogenetic network. Then $N$ is terminal planar if and only if the cut-labeled graph $L(\overline{N})$ contains no $\mathcal H^*$ structure.
\end{theorem}

This characterisation leads to the proof
Theorem~\ref{thm:terminal_forbid-te}, which claims that a phylogenetic network $N$ is terminal planar if and only if  $\overline{N}$ contains no subdivision of $K_{3,3}$, no pivotal subdivision of $K_{2:3}$, and no pivotal subdivision of ${K}_{4:2}$.

\begin{proof}[Proof of Theorem~\ref{thm:terminal_forbid-te}]
We first prove the claim that $\overline{N}$ contains a pivotal subdivision of $K_{2:3}$ if and only if  $L(\overline{N})$ contains an $(H_{2,i},h^*)$ structure for some $0\le i \le 3$

First, assume that $\overline{N}$ contains a subgraph $G$  such that $G$ is a pivotal subdivision of $K_{2:3}$. 
In other words, $G$ contains precisely five branching vertices, which can be grouped into two subsets: the first one is $V_0$ that contains three vertices $\{v_1,v_2,v_3\}$ that form the pivotal basis of $G$ in $\overline{N}$, and the other is $V_1$ that contains the other two vertices $t_1,t_2$ which are necessarily non-cut vertices of degree $3$.  
Let $i$ be the number of vertices in $V_0$ that are not cut vertices of $\overline{N}$. Without loss of generality, assume $i=2$, as the other cases can be proved similarly.
Relabelling if necessary, we may assume that $v_3$ is a cut vertex. Then  there is a pair of disjoint pivotal rays $P_1,P_2$. We can further assume that $P_i$ contains exactly one cut vertex $u_i$, which is necessarily one end of $P_i$; in particular, $P_i$ contains no cut edge.

Now consider the subgraph $F$ of $\overline{N}$ obtained as the union of $G$, $P_1$, and $P_2$.
Now consider the cut-labeled graph $L(\overline{N})=(\overline{N},\ell)$, which contains a $0/1$-labeled subgraph $(F,f)$, where $f=\ell|_F$.
Since each edge $e$ in $G$ is not a cut edge in $\overline{N}$, and neither $P_1$ nor $P_2$ contains a cut edge, it follows that $f(e)=0$ for every edge $e$ in $F$.
Let $V'=V_0\cup V_1 \cup \{u_1,u_2\}$. Then for $v\in V'$, $f(v)=1$ if and only if $v\in \{v_3,u_1,u_2\}$. 
This implies that $L(\overline{N})$ contains a $(H_{2,2},h^*)$ structure, as required.

Conversely, without loss of generality, assume that $L(\overline{N})$ contains an $(H_{2,2},h^*)$ structure. By definition, there exist a subgraph $(F,f)$ of $L(\overline{N})$ and a $0/1$-labeled graph $(F_1,f_1)$ obtained from $(F,f)$ by a finite sequence of label-preserving vertex smoothings such that $(F_1,f_1)$ is label-preserving isomorphic to $(H_{2,2},h^*)$. 
The vertices of $F_1$ are precisely the vertices of $F$ that remain after the smoothings, and they inherit their labels from $F$. Hence each edge of $H_{2,2}$ corresponds to a path in $F$. Now ignore the two paths in $F$ corresponding to the two cut edges of $H_{2,2}$. The remaining subgraph is a subdivision of $K_{2,3}$. Let $v_1,v_2,v_3$ be the three branching vertices in this remaining subgraph that correspond to the degree-$2$ vertices of $K_{2,3}$, and let $t_1,t_2$ be the other two degree-$3$ branching vertices. By the definition of $h^*$, we have $h^*(t_1)=h^*(t_2)=0$, and hence $f(t_1)=f(t_2)=0$.

Relabelling if necessary, we may assume that $h^*(v_3)=f(v_3)=1$ and that $h^*(v_k)=f(v_k)=0$ for $k=1,2$. Then $v_3$ is a cut vertex of $\overline{N}$. Moreover, for $k=1,2$, the cut edge incident with $v_k$ in $H_{2,2}$ corresponds to a path $P_k$ in $F$ between $v_k$ and a label-$1$ vertex $u_k$.
Since $f(v_k)=0$ and $f(u_k)=1$, the vertex $v_k$ is not a cut vertex of $\overline{N}$, whereas $u_k$ is a cut vertex of $\overline{N}$.
For $k=1,2$, let $Q_k$ be the subpath of $P_k$ from $v_k$ to the first cut vertex encountered along $P_k$. Since $Q_k$ is a subpath of $P_k$ for $k=1,2$, and the paths $P_1$ and $P_2$ corresponding to the distinct cut edges of $H_{2,2}$ are vertex-disjoint and meet the remaining $K_{2,3}$ subdivision only at $v_1$ and $v_2$, respectively, the paths $Q_1$ and $Q_2$ are also vertex-disjoint and meet the subdivision only at their initial vertices.

Thus $v_3$ itself gives a degenerate pivotal ray. 
Moreover, by the choice of $Q_k$, each $Q_k$ contains exactly one cut vertex as its endpoint and meets the remaining $K_{2,3}$ subdivision only at $v_k$.
Hence $Q_1$ and $Q_2$ are pivotal rays from $v_1$ and $v_2$, respectively. Therefore $\{v_1,v_2,v_3\}$ forms a pivotal basis of the above subdivision of $K_{2,3}$ in $\overline{N}$. Hence $\overline{N}$ contains a pivotal subdivision of $K_{2:3}$.

Similarly, $\overline N$ contains a pivotal subdivision of $K_{4:2}$ if and only if $L(\overline N)$ contains an $(H_{3,i},h^*)$ structure for some $0\le i\le 2$. Moreover, an $(H_1,h^*)$ structure is exactly a subdivision of $K_{3,3}$.
Together with Theorem~\ref{thm:cut-label:str}, these equivalences establish the result of the theorem. 
\end{proof}

We end this Appendix with the following observation, a slightly stronger version of Lemma~\ref{lem:H2_degree3_label0}.

\begin{lemma}\label{lem:galled_H20_H21}
Let $N$ be a galled network. Then $N$ is terminal planar if and only if $L(\overline N)$ contains neither an $(H_{2,0},h^*)$ structure nor an $(H_{2,1},h^*)$ structure.
\end{lemma}

\begin{proof}
By Lemma~\ref{lem:no_s2} and Theorem~\ref{thm:normal-Kk23}, terminal planarity of $N$ is equivalent to the absence of a normal pivotal subdivision of $K_{2:3}$ in $\overline N$. In such a subdivision, at
least two of the three paths contain reticulation vertices, which are cut vertices of $\overline N$ by Lemma~\ref{lmm:retcutvertex} and Lemma~\ref{lmm:cycle}.
Thus, among the three degree-$2$ vertices of the corresponding $H_2$ graph, at least two have label $1$. Hence the only possible $\mathcal H_2^*$ structures are those corresponding to
$(H_{2,0},h^*)$ and $(H_{2,1},h^*)$, and the claim follows from Theorem~\ref{thm:terminal_forbid-te}.
\end{proof}

\end{appendices}

\end{document}